\documentclass[12pt, reqno, twoside, letterpaper]{amsart}

\usepackage{amsmath,amssymb,amsbsy,amsfonts,amsthm,latexsym,
            amsopn,amstext,amsxtra,euscript,amscd,stmaryrd,mathrsfs,
            cite,array}
\usepackage{graphicx}
\usepackage{subcaption}
\usepackage{longtable}

\numberwithin{equation}{section}

\usepackage{color}
\usepackage[usenames,dvipsnames,svgnames,table]{xcolor}

\def\eps{\varepsilon}
\def\mand{\qquad\mbox{and}\qquad}
\def\fl#1{\left\lfloor#1\right\rfloor}

\def\({\left(}
\def\){\right)}

\newcommand{\ind}[1]{\ensuremath{\mathbf{1}_{#1}}}  
\newcommand{\e}{\ensuremath{\mathbf{e}}}

\newcommand{\cA}{\ensuremath{\mathcal{A}}}
\newcommand{\cB}{\ensuremath{\mathcal{B}}}

\newcommand{\cD}{\ensuremath{\mathcal{D}}}

\newcommand{\cF}{\ensuremath{\mathcal{F}}}

\newcommand{\cH}{\ensuremath{\mathcal{H}}}

\newcommand{\cK}{\ensuremath{\mathcal{K}}}

\newcommand{\cN}{\ensuremath{\mathcal{N}}}

\newcommand{\cS}{\ensuremath{\mathcal{S}}}

\newcommand{\cT}{\ensuremath{\mathcal{T}}}

\newcommand{\fH}{\ensuremath{\mathfrak{H}}}

\newcommand{\fS}{\ensuremath{\mathfrak{S}}}

\newcommand{\NN}{\ensuremath{\mathbb{N}}}
\newcommand{\PP}{\ensuremath{\mathbb{P}}}

\newcommand{\RR}{\ensuremath{\mathbb{R}}}
\newcommand{\ZZ}{\ensuremath{\mathbb{Z}}}

\usepackage{color}
\usepackage{hyperref}
\hypersetup{
    pdftoolbar=true,                        
    pdfmenubar=false,                        
    pdffitwindow=false,                      
    pdfstartview={FitH},                    
    pdfnewwindow=true,                       
    colorlinks=true,
    linkcolor={black},                      
    citecolor={black},
    filecolor={black},
    urlcolor={black},         
}

\RequirePackage[a4paper,top=2.54cm,bottom=2.54cm,left=1.90cm,right=1.90cm,headsep=1em,includehead,includefoot]{geometry}
\usepackage{setspace}
\usepackage{amsthm}
\newtheoremstyle{customthm}
{1em}                    
{1em}                    
{\itshape}               
{}                       
{\scshape}               
{.}                      
{5pt plus 1pt minus 1pt} 
{}                       

\newtheoremstyle{customrem}
{1em}                    
{1em}                    
{}                       
{}                       
{\scshape}               
{.}                      
{5pt plus 1pt minus 1pt} 
{}                       

\theoremstyle{customthm}

\newtheorem{X}{X}[section]
  
\newtheorem{conjecture}[X]{\bf Conjecture}  
\newtheorem{lemma}[X]{Lemma}

\newtheorem{proposition}[X]{\bf Proposition}

\theoremstyle{customrem}

\usepackage{etoolbox}
\AtEndEnvironment{remark}{\null\hfill\qedsymbol}

\AtEndEnvironment{definition}{\null\hfill\qedsymbol}

\renewcommand{\le}{\ensuremath{\leqslant}}
\renewcommand{\ge}{\ensuremath{\geqslant}}
\def\fl#1{\left\lfloor#1\right\rfloor}

\renewcommand{\pod}[1]{\mathchoice
  {\allowbreak \if@display \mkern 5mu\else \mkern 5mu\fi (#1)}
  {\allowbreak \if@display \mkern 5mu\else \mkern 5mu\fi (#1)}
  {\mkern4mu(#1)}
  {\mkern4mu(#1)}
}

\newcommand*{\defeq}{\mathrel{\vcenter{\baselineskip0.5ex \lineskiplimit0pt
                     \hbox{\scriptsize.}\hbox{\scriptsize.}}}%
                     =}

\DeclareSymbolFont{EUEX}{U}{euex}{m}{n}

\DeclareSymbolFont{euexlargesymbols}{U}{euex}{m}{n}
\DeclareMathSymbol{\intop}{\mathop}{euexlargesymbols}{"52}
     \def\int{\intop\nolimits}

\DeclareSymbolFont{euexsymbols}     {U}{euex}{m}{n}
\DeclareMathSymbol{\smallint}{\mathop}{euexsymbols}{"52}
                     
\usepackage{ifthen}
\usepackage{xifthen}

\def\sums{                        
   \@ifnextchar[
     {\sums@i}
     {\ensuremath{\sum}}    
}
\def\sums@i[#1]{
   \@ifnextchar[
     {\sums@ii{#1}}
     {\ensuremath{\sum_{#1}}}
}
\def\sums@ii#1[#2]{
   \@ifnextchar[
     {\sums@iii{#1}{#2}}
     {\ensuremath{\sum_{\substack{#1 \\ #2}}}}
}
 \def\sums@iii#1#2[#3]{
   \@ifnextchar[
     {\sums@iv{#1}{#2}{#3}}
     {\ensuremath{\sum_{\substack{#1 \\ #2 \\ #3}}}}
}
 \def\sums@iv#1#2#3[#4]{
   \@ifnextchar[
     {\sums@v{#1}{#2}{#3}{#4}}
     {\ensuremath{\sum_{\substack{#1 \\ #2 \\ #3 \\ #4}}}}
}
\def\sums@v#1#2#3#4[#5]{
     {\ensuremath{\sum_{\substack{#1 \\ #2 \\ #3 \\ #4 \\ #5}}}}
}

\def\sumss[#1]{
   \@ifnextchar[
     {\sumss@i[#1]}
     {
      \ifthenelse{\isempty{#1}} 
      {\ensuremath{\sum}}       
      {
       \ifthenelse{\equal{#1}{'}}
         {\ensuremath{\sideset{}{^{\prime}}{\sum}}}
         {\ensuremath{\sideset{}{^{#1}}{\sum}}} 
      }  
     }
}    
\def\sumss@i[#1][#2]{
   \@ifnextchar[
     {\sumss@ii[#1]{#2}}
     {
       \ifthenelse{\isempty{#1}} 
      {\ensuremath{\sum_{#2}}}       
      {
       \ifthenelse{\equal{#1}{'}}
         {\ensuremath{\sideset{}{^{\prime}}{\sum}_{#2}}}
         {\ensuremath{\sideset{}{^{#1}}{\sum}_{#2}}} 
      }  
     }
}
\def\sumss@ii[#1]#2[#3]{
   \@ifnextchar[
     {\sumss@iii[#1]{#2}{#3}}
     {
       \ifthenelse{\isempty{#1}} 
      {\ensuremath{\sum_{\substack{#2 \\ #3}}}}       
      {
       \ifthenelse{\equal{#1}{'}}
         {\ensuremath{\sideset{}{^{\prime}}{\sum}_{\substack{#2 \\ #3}}}}
         {\ensuremath{\sideset{}{^{#1}}{\sum}_{\substack{#2 \\ #3}}}}
      } 
     }
}
 \def\sumss@iii[#1]#2#3[#4]{
   \@ifnextchar[
     {\sumss@iv[#1]{#2}{#3}{#4}}
     {
       \ifthenelse{\isempty{#1}} 
      {\ensuremath{\sum_{\substack{#2 \\ #3 \\ #4}}}}       
      {
       \ifthenelse{\equal{#1}{'}}
         {\ensuremath{\sideset{}{^{\prime}}{\sum}_{\substack{#2 \\ #3 \\ #4}}}}
         {\ensuremath{\sideset{}{^{#1}}{\sum}_{\substack{#2 \\ #3 \\ #4}}}} 
      }  
     }
}
 \def\sumss@iv[#1]#2#3#4[#5]{
   \@ifnextchar[
     {\sumss@v[#1]{#2}{#3}{#4}{#5}}
     {
       \ifthenelse{\isempty{#1}} 
      {\ensuremath{\sum_{\substack{#2 \\ #3 \\ #4 \\ #5}}}}       
      {
       \ifthenelse{\equal{#1}{'}}
         {\ensuremath{\sideset{}{^{\prime}}{\sum}_{\substack{#2 \\ #3 \\ #4 \\ #5}}}}
         {\ensuremath{\sideset{}{^{#1}}{\sum}_{\substack{#2 \\ #3 \\ #4 \\ #5}}}} 
      }  
     }
}
\def\sumss@v[#1]#2#3#4#5[#6]{
     {\ifthenelse{\isempty{#1}} 
      {\ensuremath{\sum_{\substack{#2 \\ #3 \\ #4 \\ #5 \\ #6 }}}}       
      {
       \ifthenelse{\equal{#1}{'}}
         {\ensuremath{\sideset{}{^{\prime}}{\sum}_{\substack{#2 \\ #3 \\ #4 \\ #5 \\ #6 }}}}
         {\ensuremath{\sideset{}{^{#1}}{\sum}_{\substack{#2 \\ #3 \\ #4 \\ #5 \\ #6 }}}} 
      }  
     }
}

\def\sumsstxt[#1]{
   \@ifnextchar[
     {\sumsstxt@i[#1]}
     {
      \ifthenelse{\isempty{#1}} 
      {\ensuremath{\textstyle\sum}}       
      {
       \ifthenelse{\equal{#1}{'}}
         {\ensuremath{\sideset{}{^{\prime}}{\textstyle\sum}}}
         {\ensuremath{\sideset{}{^{#1}}{\textstyle\sum}}} 
      }  
     }
}    
\def\sumsstxt@i[#1][#2]{
   \@ifnextchar[
     {\sumsstxt@ii[#1]{#2}}
     {
       \ifthenelse{\isempty{#1}} 
      {\ensuremath{\textstyle\sum_{#2}}}       
      {
       \ifthenelse{\equal{#1}{'}}
         {\ensuremath{\sideset{}{^{\prime}}{\textstyle\sum}_{#2}}}
         {\ensuremath{\sideset{}{^{#1}}{\textstyle\sum}_{#2}}} 
      }  
     }
}
\def\sumsstxt@ii[#1]#2[#3]{
   \@ifnextchar[
     {\sumsstxt@iii[#1]{#2}{#3}}
     {
       \ifthenelse{\isempty{#1}} 
      {\ensuremath{\textstyle\sum_{\substack{#2 \\ #3}}}}       
      {
       \ifthenelse{\equal{#1}{'}}
         {\ensuremath{\sideset{}{^{\prime}}{\textstyle\sum}_{\substack{#2 \\ #3}}}}
         {\ensuremath{\sideset{}{^{#1}}{\textstyle\sum}_{\substack{#2 \\ #3}}}}
      } 
     }
}
 \def\sumsstxt@iii[#1]#2#3[#4]{
   \@ifnextchar[
     {\sumsstxt@iv[#1]{#2}{#3}{#4}}
     {
       \ifthenelse{\isempty{#1}} 
      {\ensuremath{\textstyle\sum_{\substack{#2 \\ #3 \\ #4}}}}       
      {
       \ifthenelse{\equal{#1}{'}}
         {\ensuremath{\sideset{}{^{\prime}}{\textstyle\sum}_{\substack{#2 \\ #3 \\ #4}}}}
         {\ensuremath{\sideset{}{^{#1}}{\textstyle\sum}_{\substack{#2 \\ #3 \\ #4}}}} 
      }  
     }
}
\def\sumsstxt@iv[#1]#2#3#4[#5]{
     {\ifthenelse{\isempty{#1}} 
      {\ensuremath{\textstyle\sum_{\substack{#2 \\ #3 \\ #4 \\ #5}}}}       
      {
       \ifthenelse{\equal{#1}{'}}
         {\ensuremath{\sideset{}{^{\prime}}{\textstyle\sum}_{\substack{#2 \\ #3 \\ #4 \\ #5}}}}
         {\ensuremath{\sideset{}{^{#1}}{\textstyle\sum}_{\substack{#2 \\ #3 \\ #4 \\ #5}}}} 
      }  
     }
}

\def\prods{              
   \@ifnextchar[
     {\prods@i}
     {\ensuremath{\prod}}    
}
\def\prods@i[#1]{
   \@ifnextchar[
     {\prods@ii{#1}}
     {\ensuremath{\prod_{#1}}}
}
\def\prods@ii#1[#2]{
   \@ifnextchar[
     {\prods@iii{#1}{#2}}
     {\ensuremath{\prod_{\substack{#1 \\ #2}}}}
}
 \def\prods@iii#1#2[#3]{
   \@ifnextchar[
     {\prods@iv{#1}{#2}{#3}}
     {\ensuremath{\prod_{\substack{#1 \\ #2 \\ #3}}}}
}
\def\prods@iv#1#2#3[#4]{
     {\ensuremath{\prod_{\substack{#1 \\ #2 \\ #3 \\ #4}}}}
}
\allowdisplaybreaks

\newcommand{\tind}[1]{\ensuremath{\widetilde{\mathbf{1}}_{#1}}} 

\title[Consecutive Piatetski-Shapiro primes]
      {Consecutive Piatetski-Shapiro primes based on the Hardy-Littlewood conjecture}

\author[Victor Zhenyu Guo]{Victor Zhenyu Guo}

\address{School of Mathematics and Statistics, Xi'an Jiaotong University, Xi'an, Shaanxi, China.}

\email{guozyv@xjtu.edu.cn}

\author[Yuan Yi]{Yuan Yi}

\address{School of Mathematics and Statistics, Xi'an Jiaotong University, Xi'an, Shaanxi, China.}

\email{yuanyi@xjtu.edu.cn}

\date{\today}

\begin{document}

\begin{abstract}
The Piatetski-Shapiro sequences are of the form $\mathcal{N}^{(c)} := (\lfloor n^c \rfloor)_{n=1}^\infty$ with $c > 1, c \not\in \mathbb{N}$. Let $p_n$ be the sequence of primes in ascending order. In this paper, we study the distribution of pairs $(p_n, p_{n+1})$ of consecutive primes such that $p_n \in \mathcal{N}^{(c_1)}$ and $p_{n+1} \in \mathcal{N}^{(c_2)}$ for $c_1, c_2 \in (1,2)$, $c_1 \neq c_2$ and give a conjecture with the prime counting functions of the pairs $(p_n, p_{n+1})$. We give a heuristic argument to support this prediction based on a model by Lemke Oliver and Soundararajan which relies on a strong form of the Hardy-Littlewood conjecture. Moreover, we prove a proposition related to the average of singular series with a weight of a complex exponential function. 
\end{abstract}

\maketitle

\begin{quote}
\end{quote}

\begin{quote}
\end{quote}



\section{Introduction}
\label{sec:intro}

The Piatetski-Shapiro sequences are sequences of the form
$$
\cN^{(c)} \defeq (\fl{n^c})_{n=1}^\infty\qquad(c>1,~c \not\in\NN).
$$
Piatetski-Shapiro~\cite{PS} proved that if $c\in(1,\frac{12}{11})$ the counting function
$$
\pi^{(c)}(x) \defeq | \big\{\text{\rm prime~}p\le x : p\in\cN^{(c)}\big\} |
$$
satisfies the asymptotic relation
$$
\pi^{(c)}(x) \sim \frac{x^{1/c}}{\log x} \qquad \text{~\rm as } x \to \infty.
$$
The admissible range for $c$ of the above formula has been extended many times and is currently known to hold for all $c\in(1,\frac{2817}{2426})$ thanks to Rivat and Sargos~\cite{RiSa}. Rivat and Wu~\cite{RiWu} also showed that there are infinitely many Piatetski-Shapiro primes for $c \in (1, \frac{243}{205})$. We refer the readers to see \cite{Guo} for more details of the improvements of $c$. The asymptotic relation is expected to hold for all values of $c \in (1,2)$. The estimation of Piatetski-Shapiro primes is an approximation of the well-known conjecture that there exist infinitely many primes of the form $n^2+1$. 

For a better understanding of the distribution of primes, it is natural to study consecutive primes, for example the twin prime conjecture. In this article, to understand the distribution of Piatetski-Shapiro primes, we are interested in the counting function of consecutive primes in Piatetski-Shapiro sequences. Let $p_n$ be the sequence of primes in ascending order. Let $\mathbf{c} \defeq (c_1, c_2, \cdots, c_r )$ with real numbers $c_i \in (1,2)$ for all integers $1 \le i \le r$. We define that
\begin{equation}
\label{eq:pic}
\pi(x;\mathbf{c}) \defeq |\{p_n \le x: p_{n+i-1} \in \cN^{(c_i)} \text{~for~} 1 \le i \le r\} |
\end{equation}
and aim at the behaviour of this counting function.

Our idea is inspired by a breakthrough in 2016 by Lemke Oliver and Soundararajan \cite{LO-S}. Let $q \ge 3$ and $\mathbf{a} \defeq (a_1, \cdots, a_r)$ with $(a_i,q) = 1$ for all $1 \le i \le r$. Applying a model based on a modified version of the Hardy-Littlewood conjecture, Lemke Oliver and Soundararajan~\cite{LO-S} investigated the biases in the occurrence of the pattern $\mathbf{a}$ in strings of $r$ consecutive primes reduced modulo $q$. In fact, they analyzed the counting function
$$
\pi(x;q,\mathbf{a}) \defeq | \{ p_n \le x : p_{n+i-1} \equiv a_i \bmod q \text{ for } 1 \le i \le r \} |.
$$

The method has been applied to analyze other consecutive sequences. David, Devin, Nam and Schlitt \cite{DDNS} applied Lemke Oliver and Soundararajan's method to study consecutive sums of two squares. Let 
$$
\mathbb{E} \defeq \{a^2+b^2: a, b \in \mathbb{Z} \} \defeq \{E_n: n \in \mathbb{N}\}. 
$$
By the Hardy-Littlewood conjectures in arithmetic progressions for sum of two squares, they gave a heuristic argument of a conjecture of the counting function 
$$
|\{E_n \le x: E_n \equiv a \pmod{q}, E_{n+1} \equiv b \pmod{q} \} |. 
$$

For any given real numbers $\alpha>0$ and  $\beta\ge 0$, the associated (generalized) Beatty sequence is defined by
$$
\cB_{\alpha,\beta}\defeq\big(\fl{\alpha m+\beta}\big)_{m\in\NN},
$$
which is also called the generalized arithmetic progression. Banks and Guo \cite{BaGu} gave a conjecture of the estimation of the counting function
$$
\big|\{p_n \le x: p_n \in \cB_{\alpha, \beta} \text{~and~} p_{n+1} \in\cB_{\hat\alpha, \hat\beta}\} \big|
$$
by a heuristic argument based on the method of Lemke Oliver and Soundararajan.

In this article, we apply a similar model to give a conjecture to the counting function \eqref{eq:pic}. In particular, the case for $r=2$ is a simple and interesting case. Fix real numbers $c_1,c_2 \in (1, 2)$. We write that
\begin{equation}
 \label{eq:PScf} 
\pi\big(x; (c_1, c_2)\big)\defeq \big|\{p_n \le x: p_n \in \cN^{(c_1)} \text{~and~} p_{n+1} \in \cN^{(c_2)} \} \big|
\end{equation}
We shall give a deatiled heuristic argument of the following conjecture in Section \ref{sec:3} with a data analysis in Section \ref{sec:6}.

\begin{conjecture}
\label{conj:main}
For any fixed positive number $\varepsilon > 0$ and $c_1 \neq c_2$, the counting function $\eqref{eq:PScf}$ satisfies that
$$
\pi\big(x; (c_1, c_2)\big) =\frac{x^{1/c_1+1/c_2-1}}{c_1 c_2 \log x} + O\(\frac{x^{1/c_1+1/c_2-1}}{(\log x)^{3/2 - \eps}}\),
$$
where the implied constant depends only on $c_1$, $c_2$ and $\eps$.
\end{conjecture}

\textbf{Remark.}
The case that $c_1 = c_2$ is easy to be excluded. By the well-believed Cram\'er's model, it follows that $p_{n+1} - p_n \ll \log^2 p_n$. If Conjecture~\ref{conj:main} holds for the case when $c_1 = c_2 = c$, let $p_{n+1}= \fl{n_1^c}, p_n = \fl{n_2^c}, c > 1.$ We derive that 
$$
p_{n+1}- p_n = \fl{n_1^c} - \fl{n_2^c} > n_1^c - 1 - n_2^c \ge (n_2+1)^c - 1 - n_2^c \gg n_2^{c-1}.
$$ 
Since $n_2 \sim {p_n}^{1/c}$, it gives that $p_{n+1} - p_n \gg {p_n}^{\frac{c-1}{c}}$, which conflicts the Cram\'er's model. 

\vspace{3cm}

In Section \ref{sec:5}, we also give a sketch of a more general conjecture.

\begin{conjecture}
\label{conj:0}
Given any fixed positive number $\eps > 0$ and $c_i \neq c_j$ for $1 \le i < j \le r$, the counting function \eqref{eq:pic} satisfies that
$$
\pi(x;\mathbf{c}) = \frac{x^{1/c_1 + \dots + 1/c_r - r + 1}}{c_1 \cdots c_r \log x} + O\bigg( \frac{x^{1/c_1 + \dots + 1/c_r - r + 1}}{(\log x)^{3/2-\eps}} \bigg),
$$
where the implied constant depends only on $c_1$, \dots, $c_r$ and $\eps$.
\end{conjecture}

In what follows we give a short survey of the breakthrough of Lemke Oliver and Soundararajan's biases. We will end the introduction by the main proposition and key improvement to this topic. 

\subsection{The Hardy-Littlewood conjecture}
  Let $\cH$ be a finite subset of $\ZZ$, and let $\ind{\PP}$ denote the indicator function of the primes. A strong form of the Hardy-Littlewood conjecture for $\cH$ asserts that the estimate
\begin{equation}
	\label{eq:HL}
	\sum_{n\le x}\prod_{h\in\cH}\ind{\PP}(n+h) =\fS(\cH)\int_2^x\frac{du}{(\log u)^{|\cH|}} + O(x^{1/2+\eps})
\end{equation}
holds for every fixed $\eps>0$, where $\fS(\cH)$ is the singular series given by
$$
\fS(\cH)\defeq\prod_p\bigg(1-\frac{|(\cH\bmod p)|}p\bigg)
\bigg(1-\frac1p\bigg)^{-|\cH|}.
$$
For their work on primes in short intervals, Montgomery and Soundararajan \cite{MontSound} have introduced the modified singular series
$$
\fS_0(\cH) \defeq \sum_{\cT \subseteq \cH} (-1)^{|\cH \setminus \cT|} \fS(\cT),
$$
for which one has the relation
$$
\fS(\cH)=\sum_{\cT\subseteq\cH}\fS_0(\cT).
$$
Note that $\fS(\varnothing)=\fS_0(\varnothing)=1$. The Hardy-Littlewood conjecture \eqref{eq:HL} can be reformulated in terms of the modified singular series as follows:
\begin{equation}
\label{eq:modHL}
\sum_{n\le x}\prod_{h\in\cH} \bigg(\ind{\PP}(n+h)-\frac{1}{\log n}\bigg) =\fS_0(\cH)\int_2^x\frac{du}{(\log u)^{|\cH|}} + O(x^{1/2+\eps}).
\end{equation}

\subsection{A modified Hardy-Littlewood conjecture with congruence conditions}

To investigate the distribution of primes in arithmetic progressions, we introduce a modification of the Hardy-Littlewood conjecture with congruence conditions $(\bmod \, q)$ from Lemke Oliver and Soundararajan's model~\cite{LO-S}. For any integer $q \ge 1$ and a finite subset $\mathcal{H} \subset \mathbb{Z}$, define the singular series away from $q$ by
$$
\mathfrak{S}_q(\mathcal{H}) \defeq \prod_{p \nmid q} \bigg( 1- \frac{|(\mathcal{H} \bmod p)|}{p} \bigg) \bigg( 1 - \frac{1}{p} \bigg)^{-|\mathcal{H}|}.
$$
We require that $a \pmod{q}$ is such that $(h+a, q) = 1$ for all $h \in \mathcal{H}$, then it asserts that
$$
\sum_{\substack{n < x \\ n \equiv a \pmod{q}}} \prod_{h \in \mathcal{H}} \mathbf{1}_{\PP} (n+h) \sim \mathfrak{S}_q(\mathcal{H}) \bigg( \frac{q}{\varphi(q)} \bigg)^{|\mathcal{H}|} \frac{1}{q} \int_2^x \frac{dy}{(\log y)^{|\mathcal{H}|}}.
$$
Now similar to $\mathfrak{S}_0$, define
$$
\mathfrak{S}_{q,0}(\mathcal{H}) \defeq \sum_{\cT \subset \cH} (-1)^{|\cH \backslash \cT|} \fS_q(\cT),
$$
which gives that
$$
\fS_q(\mathcal{H}) = \sum_{\cT \subset \cH} \fS_{q,0}(\cT).
$$
Conditioning $(h+a,q) =1$ for all $h \in \cH$, we expect that
\begin{equation}
\label{eq:apHL}
\sum_{\substack{n < x \\ n \equiv a \pmod{q}}} \prod_{h \in \mathcal{H}} \bigg( \mathbf{1}_{\PP} (n+h) -\frac{q}{\varphi(q) \log n} \bigg) \sim \mathfrak{S}_{q,0}(\mathcal{H}) \bigg( \frac{q}{\varphi(q)} \bigg)^{|\mathcal{H}|} \frac{1}{q} \int_2^x \frac{dy}{(\log y)^{|\mathcal{H}|}}.
\end{equation}
The term $q/(\varphi(q) \log n)$ is expected to be the probability that $n+h$ is prime, given that $(n+h, q) = 1$.

\subsection{Lemke Oliver and Soundararajan's method}

Based on a model by assuming a modified version of the Hardy-Littlewood conjecture~\eqref{eq:apHL}, Lemke Oliver and Soundararajan~\cite{LO-S} conjectured that
$$
\pi(x;q,\mathbf{a}) = \frac{\text{li}(x)}{\varphi(q)^r} \( 1 + c_1(q; \mathbf{a}) \frac{\log\log x}{\log x} + c_2(q;\mathbf{a}) \frac{1}{\log x} + O( (\log x)^{-7/4}) \),
$$
where $c_1(q; \mathbf{a})$ and $c_2(q;\mathbf{a})$ are explicit constants. 

They rewrited the sum of the characteristic function 
$$
\sum_{\substack{n \le x \\ n \equiv a \pmod{q}}} \mathbf{1}_{\PP}(n) \mathbf{1}_{\PP}(n+h) \prod_{\substack{0 < t < h \\ (t+a, q) = 1}} (1 - \mathbf{1}_{\PP}(n+t))
$$
into a sum related to singular series and achieved that
$$
\pi(x;q,(a,b)) \sim \frac{1}{q} \int_2^x {\alpha(y)}^{\epsilon_q(a,b)} \bigg( \frac{q}{\varphi(q) \log y} \bigg)^2 \cD(a,b;y) dy
$$
where $\cD(a,b;y)$ is a sum depending on the average of singular series. The key point to analyze $\cD(a,b;y)$ is a detailed estimation of 
$$
\sum_{\substack{h > 0 \\ h \equiv v \pmod{q}}} \mathfrak{S}_{q,0} (\{0,h\}) e^{-h/H},
$$
which was calculated as the main proposition in \cite{LO-S}. 

\subsection{Key proposition of this article}

To complete the heuristic of Conjecture~\ref{conj:main}, the key point is to analyze the average of singular series with a weight of exponential functions which differs from the case in \cite{LO-S}. Let $\nu(u) = 1 - 1/\log u$. By Lemke Oliver and Soundararajan's idea~\cite{LO-S}, one can estimate the following expression
$$
\sum_{\substack{h \ge 1 \\ 2|h}} \mathfrak{S}_0(\{0,h\}) \nu(u)^h.
$$
Since the counting function of the Piatetski-Shapiro sequence requires us to express the fractional part of a function into a sum of exponential sums, we need to estimate 
\begin{equation}
\label{eq:s}
\sum_{\substack{h \ge 1 \\ 2|h}} \mathfrak{S}_0(\{0,h\}) \nu(u)^h \e(f(h,u)),
\end{equation}
where the function $f(h,u)$ is ``smooth". The case when the function $f(h,u)$ is linear was estimated in \cite{BaGu}, but the method has to be revised to adapt the smooth case~\eqref{eq:s}. One can compare the following proposition to Lemma~2.4 in \cite{BaGu}. A detailed proof is in Section \ref{sec:4}. 

\begin{proposition}
\label{lem:RST}
Fix $\theta\in[0,1]$ and $\vartheta=0$ or $1$. Let $\gamma_1, \gamma_2 \in (0,1)$ be two real numbers. For all $j, k \in\RR$ and $u\ge 3$, let $c(j,k,u,h)$ be a complex number with $|c(j,k,u,h)| = 1$ and $c(j,k,u,h) = 1$ if $j=k=0$. We define
\begin{align*}
R_{\theta,\vartheta;j,k}(u)&\defeq \sum_{\substack{h\ge 1\\2\,\mid\,h}} h^\theta(\log h)^\vartheta\nu(u)^h c(j,k,u,h) \e(j u^{\gamma_1} + k (u+h)^{\gamma_2}),\\
S_{j,k}(u)&\defeq \sum_{\substack{h\ge 1\\2\,\mid\,h}} \fS_0(\{0,h\})\nu(u)^h c(j,k,u,h) \e(j u^{\gamma_1} + k (u+h)^{\gamma_2}).
\end{align*}
  When $j=k=0$ we have the estimates
  \begin{align*}
  R_{\theta,0;0,0}(u)&=\tfrac12\Gamma(1+\theta) (\log u)^{1+\theta}+O(1),\\
  R_{\theta,1;0,0}(u)&=\tfrac12(\log 2)\Gamma(1+\theta) (\log u)^{1+\theta}+O(1),\\
  S_{0,0}(u)&=\tfrac12\log u-\tfrac12\log\log u+O(1).
  \end{align*}
  On the other hand, if $k$ is such that $|k|\ge(\log u)^{-1}$, then
  $$
  \max\big\{|R_{\theta,\vartheta;j,k}(u)|, |S_{j,k}(u)|\big\} \ll |k|^{-4}.
  $$
  If $j,k$ are such that $|jk| \ge (\log u)^{-1}$, then
  \begin{equation}
  \label{eq:lemext}
  \max\big\{|R_{\theta,\vartheta;j,k}(u)|, |S_{j,k}(u)|\big\} \ll |jk|^{-4}.
  \end{equation}
  \end{proposition}
  
 \subsection{Remark to the main term of Conjecture \ref{conj:main}}
\label{sec:1.5}

The main conjectures by Lemke Oliver and Soundararajan \cite{LO-S} provide a secondary main term, which makes the prediction more accurate and also show the unexpected biases. However, our conjectures only prove a main term although we apply a similar model. This is because of the construction of our key Proposition~\ref{lem:RST}. 

Lemke Oliver and Soundararajan only need to consider $R_{\theta,\vartheta;j,k}(u)$ and $S_{j,k}(u)$ when $j = k =0$, which can be estimated by a careful calculation by L-functions and so on. In our problems, the way to construct Piatetski-Shapiro sequence is by the method of exponential sums. Whenever the sum is twisted by the term
$$
c(j,k,u,h) \e(j u^{\gamma_1} + k (u+h)^{\gamma_2}),
$$
we can hardly give an asymptotic formula instead of an upper bound. Therefore, unfortunately by the method in this article, we are not able to provide a secondary main term in our conjectures as the previous result.

We mention that David, Devin, Nam and Schlitt \cite{DDNS} also gave a conjecture with a secondary main term, since they formulated a new Hardy-Littlewood conjecture in arithmetic progressions for sum of two squares. If we also assume a Hardy-Littlewood conjecture for Piatetski-Shapiro primes, it may be possible to address a more accurate conjecture for consecutive Piatetski-Shapiro primes, which could be considered as a future topic. 

\section{Preliminaries}

\subsection{Notation}

We denote by $\fl{t}$ and $\{t\}$ the greatest integer $\le t$ and the fractional part of $t$, respectively. We also write $\e(t)\defeq e^{2\pi it}$ for all $t\in\RR$, as usual. We make considerable use of the sawtooth function defined by
$$
\psi(t) \defeq t-\fl{t}-\tfrac12=\{t\}-\tfrac12\qquad(t\in\RR).
$$

Let $\PP$ denote the set of primes in $\NN$. In what follows, the letter $p$ always denotes a prime number. Let $\gamma \defeq c^{-1}, \gamma_1 \defeq {c_1}^{-1}$ and $\gamma_2 \defeq {c_2}^{-1}$. Throughout the paper, $\varepsilon$ is always a sufficiently small positive number. 

For an arbitrary set $\cS$, we use $\ind{\cS}$ to denote its indicator function:
$$
\ind{\cS}(n)\defeq\begin{cases}
	1&\quad\hbox{if $n\in\cS$,}\\
	0&\quad\hbox{if $n\not\in\cS$,}\\
\end{cases}
$$
and let $\mathbf{1}^{(c)}(n) \defeq \ind{\cN^{(c)}}(n)$.

Throughout the paper, implied constants in symbols $O$, $\ll$ and $\gg$ may depend (where obvious) on the parameters $\gamma_1,\gamma_2,\eps$ but are absolute otherwise. For given functions $F$ and $G$, the notations $F\ll G$, $G\gg F$ and
$F=O(G)$ are all equivalent to the statement that the inequality $|F|\le c|G|$ holds with some constant $c>0$.

\subsection{Technical lemmas}

We start by a simple average estimation of singular series. 

\begin{lemma}
\label{lem:G0ests}
Let $h$ be a positive integer. We have
\begin{align*}
\sum_{\substack{1\le t\le h-1}}\fS_0(\{0,t\}) &\ll h^{1/2+\eps},\\
\sum_{\substack{1\le t\le h-1}}\fS_0(\{t,h\}) &\ll h^{1/2+\eps},\\
\sum_{\substack{1\le t_1<t_2\le h-1}}\fS_0(\{t_1,t_2\}) &=-\tfrac12 h\log h+\tfrac12Ah+O(h^{1/2+\eps}),
\end{align*}
where $A\defeq 2-C_0-\log 2\pi$ and $C_0$ denotes the Euler-Mascheroni constant.
\end{lemma}
\begin{proof}
See~\cite[Lemma~2.2]{BaGu}.
\end{proof}

Recall that
$$
\nu(u) = 1-\frac1{\log u}\qquad(u>1).
$$
Note that $\nu(u)$ is the same as $\alpha(u)$ in the notation of \cite{LO-S} and $\nu(u) \asymp 1$ if $u \ge 3$. 

\begin{lemma}
\label{lem:Dom}
Let $c > 0$ be a constant, and suppose that $f$ is a function such that $|f(h)|\le h^c$ for all $h\ge 1$. Then, uniformly for $3\le u\le x$ and a real function $g(u,h) \in\RR$ we have
$$
\sum_{\substack{h\le\log x\\2\,\mid\,h}} f(h)\nu(u)^h\e(g(u,h)) =\sum_{\substack{h\ge 1\\2\,\mid\,h}} f(h)\nu(u)^h\e(g(u,h))+O_c(x^{-1}).
$$
\end{lemma}
\begin{proof}
Let $H \defeq -(\log \nu(n))^{-1}$. We write $\nu(u)^h = e^{-h/H}$. By $H \le \log u$ for $u \ge 3$, $h > \log x$, we conclude that $h/H \ge h^{2/3}$ with $u \le x$. Hence
\begin{align*}
&\bigg| \sum_{\substack{h > \log x\\2\,\mid\,h}} f(h)\nu(u)^h\e(g(u,h)) \bigg| \\ &\quad \le \sum_{h > \log x} h^c e^{-h^{2/3}} \le x^{-1} \sum_{h > \log x} h^c e^{h^{1/3}-h^{2/3}} \ll_c x^{-1}.
\end{align*}
\end{proof}

\begin{lemma}
\label{lem:fh}
Assuming the Hardy-Littlewood conjecture~\eqref{eq:modHL}, let $h \le \log x$ and write
$$
f_h(n) \defeq \ind{\PP}(n)\ind{\PP}(n+h) \prod_{0<t<h}\big(1-\ind{\PP}(n+t)\big).
$$
Let
\begin{equation}
\label{eq:Sh}
S_h(x)\defeq\sum_{n\le x}f_h(n).
\end{equation}
For every integer $L\ge 0$ we denote
\begin{equation}
\label{eq:duck}
\cD_{h,L}(u)\defeq \mathop{\sum_{\cA\subseteq\{0,h\}}\sum_{\cT\subseteq[1,h-1]}} \limits_{(|\cA|+|\cT|=L)} (-1)^{|\cT|} \fS_0(\cA\cup\cT) (\nu(u)\log u)^{-|\cT|} \nu(u)^h.
\end{equation}
Then
$$
S_h(x)=\sum_{L=0}^{h+1}\int_3^x\nu(u)^{-1}(\log u)^{-2} \cD_{h,L}(u)\,du+O(x^{1/2+\eps}).
$$
\end{lemma}
\begin{proof}
First, write $\tind{\PP}(n)\defeq\ind{\PP}(n)-1/\log n$, and put
$$
\widetilde f_h(n)\defeq \(\tind{\PP}(n)+\frac{1}{\log n}\) \(\tind{\PP}(n+h)+\frac{1}{\log n}\) \prod_{0<t<h}\(1-\frac{1}{\log n}-\tind{\PP}(n+t)\).
$$
Recalling that $h\le\log x$, we have the uniform estimate
$$
f_h(n)=\widetilde f_h(n)\(1+O\(\frac{(\log x)^6}{x^{1/2}}\)\)
$$
for all $n>x^{1/2}$.  Since $f_h(n)$ and $\widetilde f_h(n)$ are bounded, it follows that
$$
S_h(x)=\sum_{n\le x}\widetilde f_h(n)+O(x^{1/2+\eps}).
$$
It follows that, up to an error term of size $O(x^{1/2+\eps})$, the quantity $S_h(x)$ equals
$$
\sum_{\cA\subseteq\{0,h\}}\sum_{\cT\subseteq[1,h-1]} (-1)^{|\cT|}\sum_{n\le x}\(\frac{1}{\log n}\)^{2-|\cA|} \(1-\frac{1}{\log n}\)^{h-1-|\cT|} \prod_{t\in\cA\cup\cT}\tind{\PP}(n+t)
$$
(compare to \cite[Equations~(2.5) and (2.6)]{LO-S}).

By the modified Hardy-Littlewood conjecture \eqref{eq:modHL} the estimate
\begin{align*}
\sum_{n\le x}(\log n)^{-c} \prod_{t\in\cH}\tind{\PP}(n+t) &=\int_{3^-}^x(\log u)^{-c}\,d\(\,\sum_{n\le u} \prod_{t\in\cH}\tind{\PP}(n+t)\)\\
&=\fS_0(\cH)\int_3^x(\log u)^{-c-|\cH|}\,du + O(x^{1/2+\eps})
\end{align*}
holds uniformly for any constant $c>0$; consequently, up to an error term of size $O(x^{1/2+\eps})$ the quantity $S_h(x)$ is equal to
$$
\sum_{\cA\subseteq\{0,h\}}\sum_{\cT\subseteq[1,h-1]} (-1)^{|\cT|} \fS_0(\cA\cup\cT) \int_3^x(\log u)^{-2-|\cT|} \nu(u)^{h-1-|\cT|}\,du.
$$
By the definition~\eqref{eq:duck}, we have
$$
S_h(x)=\sum_{L=0}^{h+1}\int_3^x\nu(u)^{-1}(\log u)^{-2}
\cD_{h,L}(u)\,du+O(x^{1/2+\eps}).
$$
\end{proof}

We need the following well known approximation of Vaaler.

\begin{lemma}
\label{lem:Vaaler}
For any $H\ge 1$ there are numbers $a_h,b_h$ such that
$$
\bigg|\psi(t)-\sum_{0<|h|\le H}a_h\,\e(th)\bigg| \le\sum_{|h|\le H}b_h\,\e(th),\qquad a_h\ll\frac{1}{|h|}\,,\qquad b_h\ll\frac{1}{H}\,.
$$
\end{lemma}
\begin{proof}
See~\cite{Vaal}.
\end{proof}

\begin{lemma}
\label{lem:ip}
Let $g(n)$ be bounded, $|a_j| \ll 1/|j|$. Then
\begin{align*}
&\sum_{3 < n \le x} \sum_{1 \le |j| \le J} g(n) a_{j} \(\e(j (n+h+1)^{\gamma}) - \e(j (n+h)^{\gamma}) \) \\
&\qquad \ll x^{\gamma-1} \max_{N \le x} \sum_{1 \le j \le J} \Bigg| \sum_{3 < n \le N} g(n) \e(j(n+h)^\gamma) \Bigg|. 
\end{align*}
\end{lemma}
\begin{proof}
Let 
$$
\phi_{j,h}(t):=\e\(j((t+h+1)^\gamma-(t+h)^\gamma)\)-1.
$$
Then 
$$
\phi_{j,h}(t)\ll |j| t^{\gamma -1} \mand \frac{\partial\phi_{j,h}(t)}{\partial t}\ll|j|t^{\gamma-2},
$$
so we have
\begin{align*}
&\quad \sum_{3 < n \le x} \sum_{1 \le |j| \le J} g(n) a_{j} \(\e(j (n+h+1)^{\gamma}) - \e(j (n+h)^{\gamma}) \) \\
&\ll \sum_{1 \le j \le J} j^{-1} \Bigg| \sum_{3 < n \le x} g(n) \phi_{j,h}(n) \e(j(n+h)^\gamma) \Bigg|\\
&\ll \sum_{1 \le j \le J} j^{-1} \Bigg| \phi_{j,h}(x) \sum_{3 < n \le x} g(n) \e(j(n+h)^\gamma) \Bigg|\\
&\qquad + \int_3^{x} \sum_{1 \le j \le J} j^{-1} \Bigg| \frac{\partial\phi_{j,h}(u)}{\partial u} \sum_{3 < n \le u} g(n) \e(j(n+h)^\gamma) \Bigg| du\\
&\ll x^{\gamma-1} \max_{N \le x} \sum_{1 \le j \le J} \Bigg| \sum_{3 < n \le N} g(n) \e(j(n+h)^\gamma) \Bigg|. 
\end{align*}
\end{proof}



Finally, we use the following well-known lemma, which provides a characterization of the numbers that occur in the Piatetski-Shapiro sequence $\cN^{(c)}$.

\begin{lemma}
\label{lem:PS}
A natural number $m$ has the form $\fl{n^c}$ if and only if $\mathbf{1}^{(c)}(m) = 1$, where
$\mathbf{1}^{(c)}(m) \defeq \fl{-m^\gamma} - \fl{-(m+1)^\gamma}$.  Moreover,
$$
\mathbf{1}^{(c)}(m)=\gamma m^{\gamma-1}+ \psi(-(m+1)^\gamma) - \psi(-m^\gamma)+O(m^{\gamma-2}).
$$
\end{lemma}

\section{Heristic of Conjecture~\ref{conj:main}}
\label{sec:3}

We assume that $c_1 \neq c_2$. For every even integer $h\ge 2$ we denote
\begin{align*}
\pi_h\big(x;(c_1, c_2)\big) &\defeq \big|\{p_n \le x: p_n \in \cN^{(c_1)}, ~p_{n+1} \in \cN^{(c_2)} \text{~and~} p_{n+1} - p_n =h\}\big| \\
&=\sum_{n\le x}\mathbf{1}^{(c_1)}(n)\mathbf{1}^{(c_2)}(n+h)f_h(n),
\end{align*}
where
$$
f_h(n) \defeq \ind{\PP}(n)\ind{\PP}(n+h) \prod_{0<t<h}\big(1-\ind{\PP}(n+t)\big).
$$
Recall that $\gamma_1 \defeq {c_1}^{-1} \in (0,1)$ and $\gamma_2 \defeq {c_2}^{-1} \in (0,1)$, clearly,
$$
\pi\big(x;(c_1,c_2)\big) =\sum_{\substack{h\le\log x\\2\,\mid\,h}}\pi_h\big(x;(c_1,c_2)\big)+O\(\frac{x^{\gamma_1+\gamma_2-1}}{(\log x)^2}\).
$$

Fixing an even integer $h\in[1,\log x]$ for the moment, our initial goal is to express $\pi_h\big(x;(c_1,c_2)\big)$ in terms of the function
$$
S_h(x)=\sum_{n\le x}f_h(n)
$$
introduced by Lemke Oliver and Soundararajan~\cite[Equation~(2.5)]{LO-S}. By Lemma~\ref{lem:PS} and Lemma~\ref{lem:Vaaler}, taking $J \defeq x^{1-\gamma_1 + \eps}$ with $a_j \ll \ 1/|j|$ and $b_j \ll 1/J$ we write
\begin{align*}
\mathbf{1}^{(c_1)}(n) &= \gamma_1 n^{\gamma_1-1}+ \( \psi(-(n+1)^{\gamma_1} - \psi(-n^{\gamma_1}) \)+O(n^{\gamma_1-2}) \\
&= \gamma_1 n^{\gamma_1-1} + \sum_{1 \le |j| \le J} a_{j} \( \e(j(n+1)^{\gamma_1} )- \e(j n^{\gamma_1}) \) \\
&\quad + O\( \sum_{0 \le |j| \le J} b_{j} \( \e(j n^{\gamma_1}) + \e(j(n+1)^{\gamma_1}) \) \) + O(n^{\gamma_1-2}).
\end{align*}
Similarly, taking that $K \defeq x^{1-\gamma_2+\eps}$ with $a_k \ll \ 1/|k|$ and $b_k \ll 1/K$ we have
\begin{align*}
\mathbf{1}^{(c_2)}(n+h) &= \gamma_2 (n+h)^{\gamma_2-1}+ \( \psi(-(n+h+1)^{\gamma_2} - \psi(-(n+h)^{\gamma_2})\)+O(n^{\gamma_2-2})\\
&=  \gamma_2 (n+h)^{\gamma_2-1} + \sum_{1 \le |k| \le K} a_{k} \( \e(k(n+h+1)^{\gamma_2}) - \e(k(n+h)^{\gamma_2}) \) \\ 
&\qquad + O\( \sum_{0 \le |k| \le K} b_{k} \( \e(k(n+h)^{\gamma_2}) + \e(k(n+h+1)^{\gamma_2}) \) \) + O (n^{\gamma_2-2}).
\end{align*}
Hence we derive the estimate
\begin{align*}
\pi\big(x;(c_1,c_2)\big) &= \sum_{\substack{h\le\log x \\ 2\, \mid \,h}} \Bigg( \cS_1 + \cS_2 + \cS_3 + \cS_4 \\
&\quad + O\( \cS_5 + \cS_6 + \cS_7 + \cS_8 + \cS_9 + \cS_{10} + \cS_{11} + \cS_{12}\) \Bigg) + O\( \frac{x^{\gamma_1+\gamma_2-1}}{\log^2 x} \),
\end{align*}
where
\begin{align*}
\cS_1 &\defeq \sum_{n \le x} \gamma_1 \gamma_2 n^{\gamma_1-1} (n+h)^{\gamma_2-1} f_h(n); \\
\cS_2 &\defeq \sum_{n \le x}  \gamma_1 n^{\gamma_1-1} f_h(n) \sum_{1 \le |k| \le K} a_{k} \(\e(k(n+h+1)^{\gamma_2}) - \e(k(n+h)^{\gamma_2}) \); \\
\cS_3 &\defeq \sum_{n \le x} \gamma_2 (n+h)^{\gamma_2-1} \sum_{1 \le |j| \le J} f_h(n) a_{j} \(\e(j (n+1)^{\gamma_1}) - \e(j n^{\gamma_1}) \); \\
\cS_4 &\defeq \sum_{n \le x} f_h(n) \( \sum_{1 \le |j| \le J} a_{j} \(\e(j(n+1)^{\gamma_1}) - \e(jn^{\gamma_1}) \) \)\\
& \qquad \cdot \( \sum_{1 \le |k| \le K} a_{k} \(\e(k(n+h+1)^{\gamma_2}) - \e(k(n+h)^{\gamma_2})  \) \); \\
\cS_5 &\defeq \sum_{n \le x}  \gamma_1 n^{\gamma_1-1} f_h(n) \sum_{0 \le |k| \le K} b_k \( \e(k(n+h)^{\gamma_2}) + \e(k(n+h+1)^{\gamma_2}) \); \\
\cS_6 &\defeq \sum_{n \le x} \gamma_2 (n+h)^{\gamma_2-1} f_h(n) \sum_{0 \le |j| \le J} b_j \( \e(jn^{\gamma_1}) + \e(j(n+1)^{\gamma_1}) \); \\
\cS_7 &\defeq \sum_{n \le x} f_h(n) \( \sum_{1 \le |j| \le J} a_j \( \e(j(n+1)^{\gamma_1}) - \e(jn^{\gamma_1}) \) \) \\
& \qquad \cdot \( \sum_{0 \le |k| \le K} b_k \( \e(k(n+h)^{\gamma_2}) + \e(k(n+h+1)^{\gamma_2}) \) \); \\
\cS_8 &\defeq \sum_{n \le x} f_h(n) \( \sum_{0 \le |j| \le J} b_j \( \e(-jn^{\gamma_1}) + \e(-j(n+1)^{\gamma_1}) \) \) \\
& \qquad \cdot \( \sum_{1 \le |k| \le K} a_k \( \e(k(n+h+1)^{\gamma_2}) - \e(k(n+h)^{\gamma_2}) \) \); \\
\cS_9 &\defeq \sum_{n \le x} f_h(n) \( \sum_{0 \le |j| \le J} b_j \( \e(jn^{\gamma_1}) + \e(j(n+1)^{\gamma_1}) \) \)\\
& \qquad \cdot \( \sum_{0 \le |k| \le K} b_k \( \e(k(n+h)^{\gamma_2}) + \e(k(n+h+1)^{\gamma_2}) \) \); \\
\cS_{10} & \defeq \sum_{n \le x} \mathbf{1}^{(c_1)}(n) n^{\gamma_2 - 2}; \\
\cS_{11} & \defeq \sum_{n \le x} \mathbf{1}^{(c_2)}(n+h) n^{\gamma_1 - 2}; \\
\cS_{12} &\defeq \sum_{n \le x} n^{\gamma_1 + \gamma_2 -4}.
\end{align*}
We work on the sums separately. 

\subsection{Estimation of $\cS_1$} 
We write $\cS_1 = \cS'_{1} + O\( \cS''_{1} \)$, where
$$
\cS'_{1} \defeq \gamma_1 \gamma_2 \sum_{n \le x} n^{\gamma_1 + \gamma_2 -2} f_h(n)
$$
and
$$
\cS''_{1} \defeq \sum_{n \le x} h n^{\gamma_1 + \gamma_2 -3} f_h(n).
$$
We consider $\cS'_{1}$. By the definition~\eqref{eq:Sh} we have
$$
\cS'_{1} = \gamma_1 \gamma_2 \int_{3^-}^x u^{\gamma_1 + \gamma_2 -2} d\(S_h(u)\) + O(\log x).
$$
Then by Lemma~\ref{lem:fh}, it follows
$$
\sum_{\substack{h \le \log x \\ 2\, \mid \,h}} \cS'_{1} =  \gamma_1 \gamma_2 \sum_{\substack{h \le \log x \\ 2\, \mid \,h}} \sum_{L=0}^{h+1}\int_3^x \frac{u^{\gamma_1 + \gamma_2 -2}}{\nu(u)\log^2 u} \cD_{h,L}(u) \,du + O(\log x).
$$
By a similar argument in~\cite[Section 2.4]{LO-S}, we conclude that the contribution of the terms with $L \ge 3$ is negligible. Taking in account Lemma~\ref{lem:Dom}, we have
\begin{equation}
\label{eq:stt}
\sum_{\substack{h \le \log x \\ 2\, \mid \,h}} \cS'_{1} = \gamma_1 \gamma_2 \sum_{l=1}^5 \int_3^x \frac{u^{\gamma_1 + \gamma_2 -2}}{\nu(u)\log^2 u} \cF_{1,l} (u) \,du + O(\log x),
\end{equation}
where
\begin{align*}
\cF_{1,1}(u) &\defeq\sum_{\substack{h\ge 1\\2\,\mid\,h}} \nu(u)^h; \\
\cF_{1,2}(u) &\defeq\sum_{\substack{h\ge 1\\2\,\mid\,h}} \fS_0(\{0,h\})\nu(u)^h; \\
\cF_{1,3}(u) &\defeq\frac{(-1)}{\nu(u)\log u} \sum_{\substack{h\ge 1\\2\,\mid\,h}} \sum_{\substack{1\le t\le h-1}} \fS_0(\{0,t\})\nu(u)^h;\\
\cF_{1,4}(u) &\defeq\frac{(-1)}{\nu(u)\log u} \sum_{\substack{h\ge 1\\2\,\mid\,h}} \sum_{\substack{1\le t\le h-1}} \fS_0(\{t,h\})\nu(u)^h;\\
\cF_{1,5}(u) &\defeq\frac{1}{(\nu(u)\log u)^2} \sum_{\substack{h\ge 1\\2\,\mid\,h}} \sum_{\substack{1\le t_1<t_2\le h-1}} \fS_0(\{t_1,t_2\})\nu(u)^h.
\end{align*}
By Lemma~\ref{lem:RST} we have
\begin{equation}
\label{eq:F11}
\cF_{1,1}(u) = R_{0,0;0,0}(u) =  \frac{1}{2} \log u + O(1)
\end{equation}
and 
\begin{equation}
\label{eq:F12}
\cF_{1,2}(u) = S_{0,0}(u) = \frac{1}{2} \log u - \frac{1}{2} \log\log u + O(1).
\end{equation}
Then combining~\eqref{eq:F11} and~\eqref{eq:F12}, we have the main term
\begin{align*}
\sum_{l=1}^2 \int_3^x \frac{u^{\gamma_1 + \gamma_2 -2}}{\nu(u)\log^2 u} \cF_{1,l} (u) \,du &= \int_3^x \frac{u^{\gamma_1 + \gamma_2 -2}}{\nu(u)\log^2 u} \( \log u - \frac{1}{2} \log \log u + O(1)\) \,du \\
&= \frac{x^{\gamma_1 + \gamma_2 - 1}}{\log x} + O\( \frac{x^{\gamma_1+\gamma_2-1}\log\log x}{\log^2 x} \).
\end{align*}

By Lemma~\ref{lem:G0ests} and Proposition~\ref{lem:RST}, we have
\begin{equation}
\label{eq:F13}
\cF_{1,3}(u) \ll \frac{1}{\log u} \sum_{\substack{h\ge 1\\2\,\mid\,h}} h^{1/2 + \eps} \nu(u)^h \ll (\log u)^{1/2 + \eps}
\end{equation}
and
$$
\cF_{1,4}(u) \ll (\log u)^{1/2 + \eps},
$$
hence for $l = 3, 4$ we get that
$$
\int_3^x \frac{u^{\gamma_1 + \gamma_2 -2}}{\nu(u)\log^2 u} \cF_{1,l} (u) \,du \ll \frac{x^{\gamma_1 + \gamma_2 -1}}{(\log x)^{3/2 - \eps}}.
$$

By Lemma~\ref{lem:G0ests}, we have
\begin{align*}
\cF_{1,5}(u) &= \frac{1}{(\nu(u)\log u)^2} \sum_{\substack{h\ge 1\\2\,\mid\,h}} \( -\tfrac12 h\log h+\tfrac12Ah+O(h^{1/2+\eps}) \)\nu(u)^h \\
& = \frac{1}{(\nu(u)\log u)^2} \( -\tfrac12 R_{1,1;0,0}(u) + \tfrac12AR_{1,0;0,0} (u) + O(R_{1/2 + \eps/2, 0; 0,0}(u)) \) \\
&\ll 1,
\end{align*}
then
$$
\int_3^x \frac{u^{\gamma_1 + \gamma_2 -2}}{\nu(u)\log^2 u} \cF_{1,5} (u) \,du \ll \int_3^x \frac{u^{\gamma_1 + \gamma_2 -2}}{\nu(u)\log^2 u} \,du \ll \frac{x^{\gamma_1 + \gamma_2 - 1}}{\log^2 x}.
$$
Therefore, we conclude that
$$
\sum_{\substack{h\le\log x \\ 2\, \mid \,h}} \cS'_{1} = \gamma_1 \gamma_2 \frac{x^{\gamma_1 + \gamma_2 - 1}}{\log x} + O\( \frac{x^{\gamma_1 + \gamma_2 -1}}{(\log x)^{3/2 - \eps}} \).
$$
By a similar method, we have
$$
\sum_{\substack{h\le\log x \\ 2\, \mid \,h}} \cS''_{1} \ll \frac{x^{\gamma_1 + \gamma_2 -1}}{(\log x)^{3/2 - \eps}}.
$$

\subsection{Estimation of $\cS_2$} After a partial summation, we apply Lemma~\ref{lem:ip} and obtain that
\begin{align*}
\cS_2 &= \gamma \sum_{3 \le n \le x} n^{\gamma_1-1} f_h(n) \sum_{1 \le |k| \le K} a_{k} \(\e(k(n+h+1)^{\gamma_2}) - \e(k(n+h)^{\gamma_2}) \) \\
&\ll x^{\gamma_2 - 1} \max_{N \le x} \sum_{1 \le k \le K} \Bigg| \sum_{3 < n \le N} n^{\gamma_1 - 1} f_h(n) \e(k(n+h)^{\gamma_2}) \Bigg|.
\end{align*} 
We define a complex function $c(k,h)$ such that
$$
\Bigg| \sum_{3 < n \le N} n^{\gamma_1 - 1} f_h(n) \e(k(n+h)^{\gamma_2}) \Bigg| = c(k,h) \sum_{3 < n \le N} n^{\gamma_1 - 1} f_h(n) \e(k(n+h)^{\gamma_2}).
$$
Note that $|c(k,h)| = 1$ and $c(k,h) = 0$ if $k=0$. Then by a similar argument to~\eqref{eq:stt}, we have
$$
\sum_{\substack{h \le \log x \\ 2\, \mid \,h}} \cS_2 \ll x^{\gamma_2 - 1} \max_{N \le x} \sum_{1 \le k \le K} \sum_{l=1}^5 \int_3^N \frac{u^{\gamma_1 - 1}}{\nu(u)\log^2 u} \cF_{2,l} (u) \, du,
$$
where
\begin{align*}
\cF_{2,1}(u) &\defeq \sum_{\substack{h \ge 1 \\ 2\, \mid \,h}} \nu(u)^h c(k,h) \e(k(u+h)^{\gamma_2}); \\
\cF_{2,2}(u) &\defeq \sum_{\substack{h \ge 1 \\ 2\, \mid \,h}} \fS_0(\{0,h\})\nu(u)^h c(k,h) \e(k(u+h)^{\gamma_2}); \\
\cF_{2,3}(u) &\defeq \frac{(-1)}{\nu(u)\log u} \sum_{\substack{h \ge 1 \\ 2\,\mid\,h}} \sum_{\substack{1\le t\le h-1}} \fS_0(\{t,h\})\nu(u)^h c(k,h) \e(k(u+h)^{\gamma_2}); \\
\cF_{2,4}(u) &\defeq \frac{(-1)}{\nu(u)\log u} \sum_{\substack{h \ge 1 \\2\,\mid\,h}} \sum_{\substack{1\le t\le h-1}} \fS_0(\{t,h\})\nu(u)^h c(k,h) \e(k(u+h)^{\gamma_2});\\
\cF_{2,5}(u) &\defeq\frac{1}{(\nu(u)\log u)^2} \sum_{\substack{h \ge 1 \\2\,\mid\,h}} \sum_{\substack{1\le t_1<t_2\le h-1}} \fS_0(\{t_1,t_2\})\nu(u)^h c(k,h) \e(k(u+h)^{\gamma_2}).
\end{align*}
By Proposition~\ref{lem:RST}, we have
$$
\max\(\cF_{2,1}(u), \cF_{2,2}(u) \) \ll |k|^{-4},
$$
provided that $|k| \ge (\log u)^{-1}$, which is sufficient that $u \ge 4$. This gives that
\begin{align*}
&\qquad x^{\gamma_2 - 1} \max_{N \le x} \sum_{1 \le k \le K} \sum_{l=1}^2 \int_3^N \frac{u^{\gamma_1 - 1}}{\nu(u)\log^2 u} \cF_{2,l} (u) \, du \\
&\ll x^{\gamma_2 - 1}\sum_{1 \le k \le K} \(1 + x^{\gamma_1} (\log x)^{-2} k^{-4}\) \ll x^{\gamma_1 + \gamma_2 - 1} (\log x)^{-2}.
\end{align*}
Similar to estimation of~\eqref{eq:F13}, by Lemma~\ref{lem:G0ests} and Proposition~\ref{lem:RST} we have
\begin{align}
\label{eq:F23}
\max\big\{ \cF_{2,3}(u), \cF_{2,4}(u) \big\} &\ll \frac{1}{\log u} \sum_{\substack{h\ge 1\\2\,\mid\,h}} h^{1/2 + \eps/2} \nu(u)^h c(k,h) \e(k(u+h)^{\gamma_2}) \nonumber \\
&\ll (\log u)^{-1} k^{-4}
\end{align}
and
\begin{align}
\label{eq:F25}
\qquad \cF_{2,5}(u) 
&= \frac{1}{(\nu(u)\log u)^2} \sum_{\substack{h\ge 1\\2\,\mid\,h}} \( -\tfrac12 h\log h + \tfrac12Ah+O(h^{1/2+\eps}) \) \nonumber \\
&\qquad \cdot \nu(u)^h c(k,h) \e(k (u+h)^{\gamma_2}) \nonumber \\
&= \frac{1}{(\nu(u)\log u)^2} \( -\tfrac12 R_{1,1;j,k}(u) + \tfrac12AR_{1,0;j,k} (u) + O(R_{1/2 + \eps/2, 0; j,k}(u)) \) \nonumber \\
&\ll (\log u)^{-2} k^{-4}. 
\end{align}
Combining~\eqref{eq:F23} and~\eqref{eq:F25}, we have
$$
x^{\gamma_2 - 1} \max_{N \le x} \sum_{1 \le l \le J} \sum_{k=3}^5 \int_3^N \frac{u^{\gamma_1 - 1}}{\nu(u)\log^2 u} \cF_{2,k} (u) \, du \ll x^{\gamma_1 + \gamma_2 -1} (\log x)^{-2}. 
$$

\subsection{Estimation of $\cS_3$} Similar to the estimation of $\cS_1$, we write $\cS_3 = \cS'_{3} + O(\cS''_{3})$, where
$$
\cS'_{3} \defeq \sum_{n \le x} \gamma_2 n^{\gamma_2-1}  \sum_{1 \le |j| \le J} f_h(n) a_{j} \(\e(j (n+1)^{\gamma_1}) - \e(j n^{\gamma_1}) \)
$$
and
$$
\cS''_{3} \defeq \sum_{n \le x} h n^{\gamma_2 - 2} \sum_{1 \le |j| \le J} f_h(n) a_{j} \(\e(j (n+1)^{\gamma_1}) - \e(j n^{\gamma_1}) \).
$$
We apply the partial summation as Lemma~\ref{lem:ip}, then
$$
\cS'_{3} \ll x^{\gamma_1 - 1} \max_{N \le x} \sum_{1 \le j \le J} \Bigg| \sum_{3 < n \le N} n^{\gamma_2 - 1} f_h(n) \e(jn^{\gamma_1}) \Bigg|.
$$
We define a complex function $c(j,h)$ such that
$$
\Bigg| \sum_{3 < n \le N} n^{\gamma_2 - 1} f_h(n) \e(jn^{\gamma_1}) \Bigg| = c(j,h) \sum_{3 < n \le N} n^{\gamma_2 - 1} f_h(n) \e(jn^{\gamma_1}).
$$
By the same construction of $\cS_2$, we conclude that
$$
\sum_{\substack{h \le \log x \\ 2\, \mid \,h}} \cS'_{3} \ll x^{\gamma_1 - 1} \max_{N \le x} \sum_{1 \le j \le J} \sum_{l=1}^5 \int_3^N \frac{u^{\gamma_2 - 1}}{\nu(u)\log^2 u} \cF_{3,l} (u) \, du,
$$
where
\begin{align*}
\cF_{3,1}(u) &\defeq \sum_{\substack{h \ge 1 \\ 2\, \mid \,h}} \nu(u)^h c(j,h) \e(ju^{\gamma_1}); \\
\cF_{3,2}(u) &\defeq \sum_{\substack{h \ge 1 \\ 2\, \mid \,h}} \fS_0(\{0,h\})\nu(u)^h c(j,h) \e(ju^{\gamma_1}); \\
\cF_{3,3}(u) &\defeq \frac{(-1)}{\nu(u)\log u} \sum_{\substack{h \ge 1 \\ 2\,\mid\,h}} \sum_{\substack{1\le t\le h-1}} \fS_0(\{t,h\})\nu(u)^h c(j,h) \e(ju^{\gamma_1}); \\
\cF_{3,4}(u) &\defeq \frac{(-1)}{\nu(u)\log u} \sum_{\substack{h \ge 1 \\2\,\mid\,h}} \sum_{\substack{1\le t\le h-1}} \fS_0(\{t,h\})\nu(u)^h c(j,h) \e(ju^{\gamma_1});\\
\cF_{3,5}(u) &\defeq\frac{1}{(\nu(u)\log u)^2} \sum_{\substack{h \ge 1 \\2\,\mid\,h}} \sum_{\substack{1\le t_1<t_2\le h-1}} \fS_0(\{t_1,t_2\})\nu(u)^h c(j,h) \e(ju^{\gamma_1}).
\end{align*}
Therefore, it follows that
$$
\sum_{\substack{h \le \log x \\ 2\, \mid \,h}} \cS'_{3} \ll x^{\gamma_2-1} \sum_{1 \le k \le K} (1 + x^{\gamma_2} (\log x)^{-2} k^{-4}) \ll x^{\gamma_1+\gamma_2-1} (\log x)^{-2}.
$$
Similarly, we know that
$$
\sum_{\substack{h \le \log x \\ 2\, \mid \,h}} \cS''_{3} \ll x^{\gamma_1 + \gamma_2 -1} (\log x)^{-2}.
$$

\subsection{Estimation of $\cS_4$} We apply Lemma~\ref{lem:ip}.
\begin{align*}
\cS_4 &\ll x^{\gamma_1 - 1} \max_{N \le x} \sum_{1 \le j \le J} \Big|\sum_{3 < n \le N} \big( f_h(n) \e(jn^\gamma_1)  \\
&\qquad \cdot \sum_{1 \le |k| \le K} a_{l} \(\e(k(n+h+1)^{\gamma_2}) - \e(k(n+h)^{\gamma_2})  \)  \big) \Big| \\
&\ll x^{\gamma_1 + \gamma_2 - 2} \max_{N \le x} \sum_{1 \le j \le J} \sum_{1 \le k \le K} \big| \sum_{3 < n \le N} f_h(n) \e(jn^{\gamma_1} + k(n+h)^{\gamma_2}) \big|.
\end{align*}
By the same method of $\cS_2$, we have
\begin{align*}
\sum_{\substack{h \le \log x \\ 2\, \mid \,h}} \cS_4 &\ll x^{\gamma_1 + \gamma_2 - 2} \max_{N \le x} \sum_{1 \le j \le J} \sum_{1 \le k \le K} \bigg( 1 + \int_4^N \frac{1}{\nu(u)\log^2 u} (jk)^{-4} \, du \\
&\qquad + \int_3^N \frac{1}{\nu(u)\log^2 u} \big( (\log u)^{-1} (jk)^{-4} + (\log u)^{-2} (jk)^{-4} \big) \, du \bigg) \\
&\ll x^{\gamma_1 + \gamma_2 - 2} \bigg(JK + x(\log x)^{-2}  + x(\log x)^{-3} + x(\log x)^{-4} \bigg) \\
&\ll \frac{x^{\gamma_1 + \gamma_2 -1}}{(\log x)^{2}}.	
\end{align*}

\subsection{Estimation of $\cS_5$} We show that
\begin{equation}
\label{eq:S5}
\sum_{\substack{h \le \log x \\ 2\, \mid \,h}} \cS_5 \ll x^{\gamma_1 + \gamma_2 - 1}.
\end{equation}
The contribution from $k = 0$ of the left-hand side of~\eqref{eq:S5} is 
\begin{equation}
\label{eq:S50}
\ll \sum_{\substack{h \le \log x \\ 2\, \mid \,h}} \sum_{n \le x} \gamma_1 n^{\gamma_1 - 1} f_h(n) K^{-1} \ll x^{\gamma_1} K^{-1} \ll x^{\gamma_1 + \gamma_2 -1 -\eps }.
\end{equation}
By the same estimation of $\cS_2$, the contribution from $k \neq 0$ of the left-hand side of~\eqref{eq:S5} is
$$
\ll K^{-1} \sum_{\substack{h \le \log x \\ 2\, \mid \,h}} \sum_{n \le x} n^{\gamma_1 - 1} f_h(n) \sum_{1 \le k \le K} \e(l(n + h)^{\gamma_2}) \ll x^{\gamma_1 + \gamma_2 - 1 -\eps}. 
$$

\subsection{Estimation of $\cS_6$} We write $\cS_6 = \cS_{61} + O(\cS_{62})$, where
$$
\cS_{61} \defeq \sum_{n \le x} \gamma_2 n^{\gamma_2-1} f_h(n) \sum_{0 \le |j| \le J} b_j \( \e(jn^{\gamma_1}) + \e(j(n+1)^{\gamma_1}) \)
$$
and
$$
\cS_{62} \defeq \sum_{n \le x} \gamma_2 h n^{\gamma_2-2} f_h(n) \sum_{0 \le |j| \le J} b_j \( \e(jn^{\gamma_1}) + \e(j(n+1)^{\gamma_1}) \).
$$
We give a brief proof of the bound
\begin{equation}
\label{eq:S61}
\sum_{\substack{h \le \log x \\ 2\, \mid \,h}} \cS_{61} \ll x^{\gamma_1 + \gamma_2 - 1 - \eps}
\end{equation} 
only, since the bound of 
$$
\sum_{\substack{h \le \log x \\ 2\, \mid \,h}} \cS_{62} \ll x^{\gamma_1 + \gamma_2 - 1 - \eps}
$$
can be derived by the same way. By a similar argument of~\eqref{eq:S50}, the contribution from $j = 0$ of the left-hand side of~\eqref{eq:S61} is
$$
\ll \sum_{\substack{h \le \log x \\ 2\, \mid \,h}} \sum_{n \le x} n^{\gamma_2-1} f_h(n) J^{-1} \ll x^{\gamma_2} J^{-1} \ll x^{\gamma_1 + \gamma_2 - 1 -\eps}.
$$
Similar to the estimation of $\cS_{31}$, the contribution form $j \neq 0$ of the left-hand side~\eqref{eq:S61} is
$$
\ll J^{-1} \sum_{\substack{h \le \log x \\ 2\, \mid \,h}} \sum_{n \le x} n^{\gamma_2 - 1} f_h(n) \sum_{1 \le j \le J} \e(j n^{\gamma_1}) \ll x^{\gamma_1 + \gamma_2 - 1 -\eps}. 
$$

\subsection{Estimation of $\cS_7$ and $\cS_8$} We prove that
\begin{equation}
\label{eq:S7}
\sum_{\substack{h \le \log x \\ 2\, \mid \,h}} \cS_7 \ll x^{\gamma_1 + \gamma_2 - 1 -\eps}.
\end{equation}
The contribution from $k = 0$ of the left-hand side of~\eqref{eq:S7} is
\begin{align*}
&= b_0 \sum_{\substack{h \le \log x \\ 2\, \mid \,h}} \sum_{n \le x} f_h(n) \( \sum_{1 \le |j| \le J} a_j \( \e(j(n+1)^{\gamma_1}) - \e(jn^{\gamma_1}) \) \) \\
&\ll J^{-1} x^{\gamma_1 - 1} \sum_{\substack{h \le \log x \\ 2\, \mid \,h}} \max_{N \le x} \bigg| \sum_{3 < n \le N} f_h(n) \e(j n^{\gamma_1}) \bigg| \ll x^{\gamma_1 + \gamma_2 - 1 - \eps},
\end{align*}
by the same estimation of $\cS_{31}$. The contribution from $k \neq 0$ of the left-hand side of~\eqref{eq:S7} is
\begin{align*}
&\ll \sum_{\substack{h \le \log x \\ 2\, \mid \,h}} \bigg| \sum_{n \le x} f_h(n) \( \sum_{1 \le |j| \le J} a_j \( \e(j(n+1)^{\gamma_1}) - \e(jn^{\gamma_1}) \) \) \nonumber \\
&\qquad \cdot \(\sum_{1 \le k \le K} b_k \e(l(n+h)^{\gamma_2}) \) \bigg| \nonumber \\
&\ll K^{-1} x^{\gamma_1-1} \sum_{\substack{h \le \log x \\ 2\, \mid \,h}} \sum_{1 \le j, l \le J} \max_{N \le x} \bigg| \sum_{3 < n \le N} f_h(n) \e(jn^{\gamma_1} + l(n+h)^{\gamma_2}) \bigg| \nonumber \\
&\ll x^{\gamma_1 + \gamma_2 - 1 - \eps},
\end{align*}
by the same estimation of $\cS_4$. The estimation of $\cS_8$ is similar. 

\subsection{Estimation of $\cS_9$} We prove that
\begin{equation}
\label{eq:S9}
\sum_{\substack{h \le \log x \\ 2\, \mid \,h}} \cS_9 \ll x^{\gamma_1 + \gamma_2 - 1 -\eps}.
\end{equation}
The contribution from $j = k = 0$ of the left-hand side of~\eqref{eq:S9} is
$$
\ll (JK)^{-1} \sum_{\substack{h \le \log x \\ 2\, \mid \,h}} \sum_{n \le x} f_h(n) \ll x^{\gamma_1 + \gamma_2 - 1 -\eps},
$$
by the trivial bound. The contribution from $j = 0$ and $k \neq 0$ of the left-hand side of~\eqref{eq:S9} is
$$
\ll (JK)^{-1} \sum_{\substack{h \le \log x \\ 2\, \mid \,h}} | \sum_{n \le x} f_h(n) \sum_{1 \le l \le J} \e(l (n+h)^{\gamma_2}) | \ll  x^{\gamma_1 + \gamma_2 - 1 -\eps},
$$
by the same estimation of $\cS_2$. The contribution from $j \neq 0$ and $k = 0$ of the left-hand side of~\eqref{eq:S9} is
$$
\ll (JK)^{-1} \sum_{\substack{h \le \log x \\ 2\, \mid \,h}} | \sum_{n \le x} f_h(n) \sum_{1 \le j \le J} \e(j n^{\gamma_2}) | \ll  x^{\gamma_1 + \gamma_2 - 1 -\eps},
$$
by the same estimation of $\cS_{31}$. The contribution from $j \neq 0$ and $l \neq 0$ of the left-hand side of~\eqref{eq:S9} is
$$
\ll (JK)^{-1} \sum_{\substack{h \le \log x \\ 2\, \mid \,h}} | \sum_{n \le x} f_h(n) \sum_{1 \le j,l \le J} \e(j n^{\gamma_1} + l (n+h)^{\gamma_2}) | \ll  x^{\gamma_1 + \gamma_2 - 1 -\eps},
$$
by the same estimation of $\cS_4$. 

\subsection{Estimation of $\cS_{10}$, $\cS_{11}$ and $\cS_{12}$}

We proceed by trivial bounds of the characteristic function:

\begin{align*}
&\sum_{\substack{h\le\log x \\ 2\, \mid \,h}} \cS_{10} \ll \sum_{\substack{h\le\log x \\ 2\, \mid \,h}} \sum_{n \le x} n^{\gamma_2 - 2} \ll x^{\gamma_2 - 1} \log^3(x); \\
&\sum_{\substack{h\le\log x \\ 2\, \mid \,h}} \cS_{11} \ll \sum_{\substack{h\le\log x \\ 2\, \mid \,h}} \sum_{n \le x} n^{\gamma_1 - 2} \ll x^{\gamma_1 - 1} \log^3(x); \\
&\sum_{\substack{h\le\log x \\ 2\, \mid \,h}} \cS_{12} \ll \sum_{\substack{h\le\log x \\ 2\, \mid \,h}} \sum_{n \le x} n^{\gamma_1 + \gamma_2 - 4} \ll x^{\gamma_1 + \gamma_2 - 3} \log^3(x),
\end{align*}

which are covered by the error term. 

\section{Proof of the key proposition}
\label{sec:4}

The proof of Proposition \ref{eq:lemext} starts by a similar construction of the proof of~\cite[Lemma~2.4]{BaGu}. Note that $\nu(u)\asymp 1$ for $u\ge 3$. Let $H \defeq -(\log \nu(n))^{-1}$, which gives that $\nu(u)^h = e^{-h/H}$. Write that
$$
\nu(u)^{h} \e(kh) = e^{-h/H_k} \qquad \text{with} \quad H_k \defeq \frac{H}{1-2\pi i k H}.
$$
Since $\Re (h/H_k) = h/H > 0$ for any positive integer $h$, by the Cahen-Mellin integral it gives that
\begin{align*}
R_{\theta,\vartheta;j,k}(u) &= \sum_{\substack{h\ge 1\\2\,\mid\,h}}	h^\theta(\log h)^\vartheta f(j,k,u,h) e^{-h/H_k}\\
&= \frac{1}{2\pi i} \int_{4-i\infty}^{4+\infty} \bigg( \sum_{\substack{h\ge 1\\2\,\mid\,h}} \frac{h^\theta(\log h)^\vartheta}{h^s} f(j,k,u,h) \bigg) \Gamma(s) H_k^s \, ds,
\end{align*}
where
$$
f(j,k,u,h) \defeq c(j,k,u,h) \e(j u^{\gamma_1} + k (u+h)^{\gamma_2} - kh).
$$
The case that $j=k=0$ is the same as \cite[Lemma~2.4]{BaGu}. When $k \neq 0$ we have
\begin{align*}  
|R_{\theta,0;j,k}(u)| &\le \frac{1}{2\pi} \int_{4-i\infty}^{4+\infty} \bigg( \sum_{\substack{h\ge 1\\2\,\mid\,h}} \big| \frac{h^\theta(\log h)^\vartheta}{h^s} \big|\bigg) \Gamma(s) H_k^s \, ds \\
&\le \frac{2^\theta}{2\pi} \int_{4-i\infty}^{4+i\infty} |2^{-4}\zeta(4-\theta)||
\Gamma(s)H_k^s | \,ds  \\
&\le\frac{2^{\theta-4}|H_k|^4}{2\pi}
\int_{-\infty}^\infty\big|\zeta(4-\theta)\Gamma(4+it)\big|\,dt\\
&\ll |H_k|^4=\bigg(\frac{H^2}{1+4\pi^2 k^2H^2}\bigg)^2,
\end{align*}
which gives that $R_{\theta,0;j,k}(u) \ll k^{-4}$ if $|k| \ge (\log u)^{-1}$ since $H \asymp \log u$ for $u \ge 3$. The bound for $R_{\theta,1;j,k}$ is proved similarly by considering $\zeta'(4-\theta)$. 
  
Secondly, we define 
$$
T_{j,k} (u) \defeq \sum_{h \ge 1} \mathfrak{S}(\{0,h\}) f(j,k,u,h) e^{-h/H_k},
$$
for $j,k \in \RR$ and $u \ge 3$. Since $\fS_0(\{0,h\})=\fS(\{0,h\})-1$ for all integers $h$,
and $\fS(\{0,h\})=0$ if $h$ is odd, it follows that
$$
S_{j,k}(u)=T_{j,k}(u)-R_{0,0;j,k}(u) = T_{j,k}(u) + O(\log u).
$$
Hence, to complete the proof of the lemma, it suffices to show that
$$
T_{j,k}(u)\ll k^{-4} \text{~if~} |k|\ge(\log u)^{-1},
$$
since the case that $j=k=0$ is the same as \cite[Lemma~2.4]{BaGu}. As in the proof of \cite[Proposition~2.1]{LO-S}, we consider the Dirichlet series
$$
F(s)\defeq\sum_{h\ge 1}\frac{\fS(\{0,h\})}{h^s},
$$
which can be expressed in the form
$$
F(s)=\frac{\zeta(s)\zeta(s+1)}{\zeta(2s+2)} \prod_p\(1-\frac{1}{(p-1)^2}+\frac{2p}{(p-1)^2(p^{s+1}+1)}\),
$$
and the final product is analytic for $\Re(s)>-1$. Similar to the proof of the first part of the lemma, using the Cahen-Mellin integral with $k\neq 0$ we have that
$$
\big|T_{j,k}(u)\big| \le \frac{|H_k|^4}{2\pi} \int_{-\infty}^\infty\big|F(4)\Gamma(4+it)\big|\,dt \ll |H_k|^4=\bigg(\frac{H^2}{1+4\pi^2 k^2 H^2}\bigg)^2.
$$
Hence $T_{j,k}(u) \ll k^{-4}$ by a similar argument. 
  
To prove \eqref{eq:lemext}, we choose 
$$
\nu(u)^{h} \e(jkh) = e^{-h/H_{j,k}} \qquad \text{with} \quad H_{j,k} \defeq \frac{H}{1-2\pi i j k H},
$$
and
$$
f(j,k,u,h) \defeq c(j,k,u,h) \e(j u^{\gamma_1} + k (u+h)^{\gamma_2} - jkh).
$$
Everything else follows the same. 
  
\section{A sketch of the heuristic argument for $r \ge 3$}
\label{sec:5}

Let $\cH \defeq (h_1, h_2, \cdots, h_{r-1})$ with $h_i$ positive even integers for $1 \le i \le r-1$. Writing that $h_0 \defeq 0$ and 
$$
\mathfrak{H}_i \defeq h_0 + \cdots + h_i,
$$
we give a brief sketch by writing $\pi(x; \mathbf{c})$ as 
$$
\sum_{n \le x} \sum_{h_1, \dots, h_{r-1} > 0} f_{\cH}(n) \prod_{i=1}^r \mathbf{1}^{(c_i)}(n + \fH_{i-1}), 
$$
where
$$
f_{\cH}(n) \defeq \mathbf{1}_{\mathbb{P}}(n) \prod_{i=1}^{r-1} \bigg[ \mathbf{1}_{\PP}(n + \fH_i) \prod_{0 < t < h_i} \big(1- \mathbf{1}_{\PP} ( n + t + \fH_{i-1} )\big) \bigg].
$$
Based on the arguments for $r = 2$ we mainly need to consider the following two sums, which are
$$
\sum_{\substack{h_1, \dots, h_{r-1} \le \log x \\ 2|h_i}} \cT_1 \quad \text{and} \quad \sum_{\substack{h_1, \dots, h_{r-1} \le \log x \\ 2|h_i}} \cT_2
$$
where
$$
\cT_1 \defeq  \sum_{n \le x} f_{\cH}(n) \prod_{i=1}^{r} \gamma_{i} (n+h_{i-1})^{\gamma_i - 1} 
$$
and
\begin{align*}
\cT_2 &\defeq \sum_{n \le x} f_{\cH}(n) \prod_{i=1}^{r} \bigg( \sum_{1 \le |k_i| \le K_i} a_{k_i} \Big(\e\big(k_i (n+1+\fH_{i-1})^{\gamma_i}\big) - \e\big(k_i (n+ \fH_{i-1})^{\gamma_i}\big) \Big) \bigg),
\end{align*}
where $K_i \defeq x^{1-\gamma_{i} + \eps}$, since $\cT_1$ provides the main term and $\cT_2$ is the most complicated error term. 

\subsection{Estimation of $\cT_1$} Based on the estimation of $\cS_1$, it is deduced that
$$
\cT_1' \defeq \sum_{n \le x} \gamma_1 \cdots \gamma_r n^{\gamma_1 + \dots + \gamma_r - r} f_{\cH}(n)
$$
is the main term and all the rest are covered by the error term. Following a similar argument in \cite[Section~4]{LO-S}, with an error term up to $O((x^{\gamma_1 + \dots + \gamma_r - r + 1})/\log^{3/2-\epsilon}x)$, we deduce that
$$
\sum_{\substack{h_1, \dots, h_{r-1} \le \log x \\ 2|h_i}} \cT_1' = \gamma_1 \cdots \gamma_r \int_{3}^{x}\frac{u^{\gamma_1 + \dots + \gamma_r - r}}{\nu(u)\log^r u}\big(\cF^{r}_{1,1}(u) + \cF^{r}_{1,2}(u)\big)~du
$$
where 
\begin{align*}
&\cF^{r}_{1,1}(u) \defeq \sum_{\substack{h_1, \dots, h_{r-1} \le \log u \\ 2|h_i}} \nu(u)^{h_1 + \cdots + h_{r-1}};\\
&\cF^{r}_{1,2}(u) \defeq \sum_{\substack{h_1, \dots, h_{r-1} \le \log u \\ 2|h_i}} \sum_{0 \leqslant s<t\leqslant r-1}\fS_0(\{0,h_{s+1} + \cdots + h_t \})\nu(u)^{h_{s+1} + \cdots + h_t}.
\end{align*}
By a similar argument to Proposition~\ref{lem:RST} we deduce that $\cT_1'$ gives the main term which is 
$$
\sum_{\substack{h_1, \dots, h_{r-1} \le \log x \\ 2|h_i}} \cT_1' = \frac{x^{1/c_1 + \dots + 1/c_r - r + 1}}{c_1 \cdots c_r \log x} + O \bigg(\frac{x^{1/c_1 + \dots + 1/c_r - r + 1}}{(\log x)^{3/2-\eps}}\bigg).
$$

\subsection{Estimation of $\cT_2$} By a partial summation similar to the estimation of $\cS_4$, it follows that
$$
\cT_2 \ll x^{\gamma_1 + \dots + \gamma_r - r} \max_{N \le x} \sum_{1 \le k_i \le K_i, \forall i \in \{1, \dots, r\} } \bigg| \sum_{3 < n \le N} f_{\cH}(n) \e\Big(\sum_{i=1}^{r} k_i (n + \fH_{i-1})^{\gamma_i} \Big) \bigg|.
$$
We define a complex function $c(\cK,\cH)$ with $\cK\defeq (k_1,\cdots,k_r)$ for $1 \le k_i \le K_i, \forall i=1,\cdots,r$ and $\cH$ as before such that
\begin{align*}
&\bigg| \sum_{3 < n \le N} f_{\cH}(n) \e\Big(\sum_{i=1}^{r} k_i (n + \fH_{i-1})^{\gamma_i} \Big) \bigg| \\
&\qquad = c(\cK,\cH) \sum_{3 < n \le N} f_{\cH}(n) \e\Big(\sum_{i=1}^{r} k_i (n + \fH_{i-1})^{\gamma_i} \Big).
\end{align*}
Similar with above we achieve that
$$
\sum_{\substack{h_1, \dots, h_{r-1} \le \log x \\ 2|h_i}} \cT_2 \ll  x^{\gamma_1 + \dots + \gamma_r - r} \max_{N \le x} \sum_{1 \le k_i \le K_i, \forall i \in \{1, \dots, r\} } \int_{3}^{N} \frac{u^{\gamma_1 + \dots + \gamma_r - r}}{\nu(u)\log^r u}\big(\cF^{r}_{2,1}(u) + \cF^{r}_{2,2}(u)\big)~du
$$
with an error term up to $O((x^{\gamma_1 + \dots + \gamma_r - r + 1})/\log^{3/2-\epsilon}x)$ where
\begin{align*}
&\cF^{r}_{2,1}(u) \defeq \sum_{\substack{h_1, \dots, h_{r-1} \le \log u \\ 2|h_i}} \nu(u)^{h_1 + \cdots + h_{r-1}}c(\cK,\cH)\e\Big(\sum_{i=1}^{r} k_i (n + \fH_{i-1})^{\gamma_i} \Big)\\
&\cF^{r}_{2,2}(u) \defeq \sum_{\substack{h_1, \dots, h_{r-1} \le \log u \\ 2|h_i}} \sum_{0 \leqslant s<t\leqslant r-1}\fS_0(\{0,h_{s+1} + \cdots + h_t \}) \\
&\qquad\qquad\qquad \qquad \qquad \cdot \nu(u)^{h_{s+1} + \cdots + h_t} \e\Big(\sum_{i=s+1}^{t} k_i (u + \fH_{i-1} )^{\gamma_i} \Big).
\end{align*}
By a similar argument to Proposition~\ref{lem:RST} we deduce that  
$$
\sum_{\substack{h_1, \dots, h_{r-1} \le \log x \\ 2|h_i}} \cT_2 \ll \frac{x^{1/c_1 + \dots + 1/c_r - r + 1}}{(\log x)^{3/2-\eps}}.
$$
Assembling these contributions yields our Conjecture~\ref{conj:0}.

\section{Comparison of Conjecture~\ref{conj:main} with numerical data}
\label{sec:6}

The main conjecture~\ref{conj:main} we provided is actually
$$
\pi\big(x; (c_1, c_2)\big) \sim \frac{x^{1/c_1 + 1/c_2 -1}}{c_1 c_2 \log x}
$$
for $c_1 \neq c_2$. Based on an algorithm written in Java, we provide numerical data to compare our conjecture with the actual values. The code can be found on the first author's website: 

\begin{center}
\vspace{0.1in}
https://sites.google.com/view/guozyv
\vspace{0.1in}
\end{center}

Indeed, we list the data as following. The error is the difference between the actual value and the predicted value while the accurate rate is the ratio between the error and the actual value. 

\begin{center}
  \begin{longtable}{ |c|c|c|c|c| } 
  \caption{The comparison between the actual value and the predicted value} \\
   
   \hline
   \vphantom{\Big|} $x$ & $\pi\big(x; (1.05, 1.03) \big)$ & Conjecture~\ref{conj:main} & error & accurate rate  \\ 
   \hline 
   $1 \cdot 10^9$ & $10426864$  & $9095109$ & $1331755$ & $12.77\%$ \\
   $2 \cdot 10^9$ & $19095275$  & $16689632$  & $2405643$ & $12.60\%$ \\
   $3 \cdot 10^9$ & $27219002$ & $23816528$ & $3402474$ & $12.50\%$ \\ 
   $4 \cdot 10^9$ & $35006715$ & $30657784$  & $4348931$ & $12.42\%$ \\
   $5 \cdot 10^9$ & $42566178$ & $37295135$ & $5271043$ & $12.38\%$ \\
   $6 \cdot 10^9$ & $49937752$ & $43774935$ & $6162817$ & $12.34\%$ \\
   $6 \cdot 10^9$ & $49937752$ & $43774935$ & $6162817$ & $12.34\%$ \\
   $7 \cdot 10^9$ & $57164363$ & $50126936$ & $7037427$ & $12.31\%$ \\
   $8 \cdot 10^9$ & $64277619$ & $56371789$ & $7905830$ & $12.30\%$ \\
   $9 \cdot 10^9$ & $71271261$ & $62524631$ & $8746630$ & $12.27\%$ \\
   $1 \cdot 10^{10}$ & $78175966$ & $68597008$ & $9578958$ & $12.25\%$ \\
   $1.1 \cdot 10^{10}$ & $84997244$ & $74598005$ & $10399239$ & $12.23\%$ \\
   $1.2 \cdot 10^{10}$ & $91746464$ & $80534947$ & $11211517$ & $12.22\%$ \\
   $1.3 \cdot 10^{10}$ & $98422803$ & $86413858$ & $12008945$ & $12.20\%$ \\
   $1.4 \cdot 10^{10}$ & $105041524$ & $92239775$ & $12801749$ & $12.19\%$ \\
   $1.5 \cdot 10^{10}$ & $111601884$ & $98016969$ & $13584915$ & $12.17\%$ \\
   $1.6 \cdot 10^{10}$ & $118112030$ & $103749105$ & $14362925$ & $12.16\%$ \\
   $1.7 \cdot 10^{10}$ & $124570606$ & $109439359$ & $15131247$ & $12.15\%$ \\
   $1.8 \cdot 10^{10}$ & $130987775$ & $115090511$ & $15897264$ & $12.14\%$ \\
   $1.9 \cdot 10^{10}$ & $137360696$ & $120705011$ & $16655685$ & $12.13\%$ \\
   $2.0 \cdot 10^{10}$ & $143694021$ & $126285036$ & $17408985$ & $12.12\%$ \\ [0.03in]

   \hline 
   \hline

  \vphantom{\Big|}  $x$ & $\pi\big(x; (1.05, 1.11) \big)$ & Conjecture~\ref{conj:main} & error & accurate rate  \\ 
   \hline 

$1 \cdot 10^9$ & $2922252$ &	$2367045$ & $555207$ & $19.00\%$ \\
$2 \cdot 10^9$ &	$5123813$ &	$4161471$ &	$962342$ &	$18.78\%$ \\
$3 \cdot 10^9$ &	$7124372$ &	$5791608$ &	$1332764$ &	$18.71\%$ \\
$4 \cdot 10^9$ &	$9000888$ &	$7323900$ &	$1676988$ &	$18.63\%$ \\
$5 \cdot 10^9$ &	$10794070$ &	$8787522$ &	$2006548$ &	$18.59\%$ \\
$6 \cdot 10^9$ &	$12521130$ &	$10198768$ &	$2322362$ &	$18.55\%$ \\
$7 \cdot 10^9$ &	$14196850$ &	$11567971$ &	$2628879$ &	$18.52\%$ \\
$8 \cdot 10^9$ &	$15830516$ &	$12902235$ &	$2928281$ &	$18.50\%$ \\
$9 \cdot 10^9$ &	$17424909$ &	$14206723$ &	$3218186$ &	$18.47\%$ \\
$1 \cdot 10^{10}$ &	$18987985$ &	$15485347$ &	$3502638$ &	$18.45\%$ \\
$1.1 \cdot 10^{10}$ &	$20523404$ &	$16741162$ &	$3782242$ &	$18.43\%$ \\
$1.2 \cdot 10^{10}$ &	$22033647$ &	$17976622$ &	$4057025$ &	$18.41\%$ \\
$1.3 \cdot 10^{10}$ &	$23520834$ &	$19193761$ &	$4327073$ &	$18.40\%$ \\
$1.4 \cdot 10^{10}$ &	$24987452$ &	$20394158$ &	$4593294$ &	$18.38\%$ \\
$1.5 \cdot 10^{10}$ &	$26433375$ &	$21579312$ &	$4854063$ &	$18.36\%$ \\
$1.6 \cdot 10^{10}$ &	$27863861$ &	$22750398$ &	$5113463$ &	$18.35\%$ \\
$1.7 \cdot 10^{10}$ &	$29278419$ &	$23908454$ &	$5369965$ &	$18.34\%$ \\
$1.8 \cdot 10^{10}$ &	$30677475$ &	$25054389$ &	$5623086$ &	$18.33\%$ \\
$1.9 \cdot 10^{10}$ &	$32062164$ &	$26188997$ &	$5873167$ &	$18.32\%$ \\
$2.0 \cdot 10^{10}$ &	$33432166$ &	$27312984$ &	$6119182$ &	$18.30\%$ \\ [0.03in]

   \hline 
      
  \end{longtable}
\end{center}

We also compare the actual values and the conjecture by the figures. The $x$-axis is the range of $x$ from $0$ to $10^9$ and the $y$-axis is the value of the actual value or the conjecture value from $0$ to $1.2 \cdot 10^7$. We use the blue line for the actual values and the red line for the conjectured values.

\begin{figure}[htb]
 \centering
  \begin{minipage}{0.48\linewidth}
  \centering
  \includegraphics[width=\linewidth]{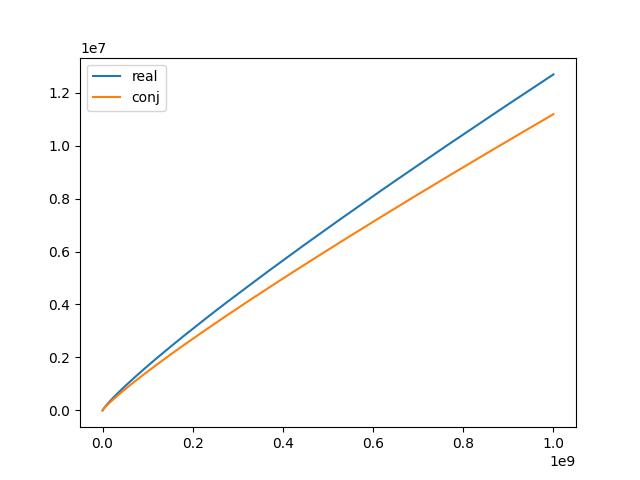}
  \caption{$c_1=1.02, c_2 = 1.05$}
  \end{minipage}\hfill
  \begin{minipage}{0.48\linewidth}
  \centering
  \includegraphics[width=\linewidth]{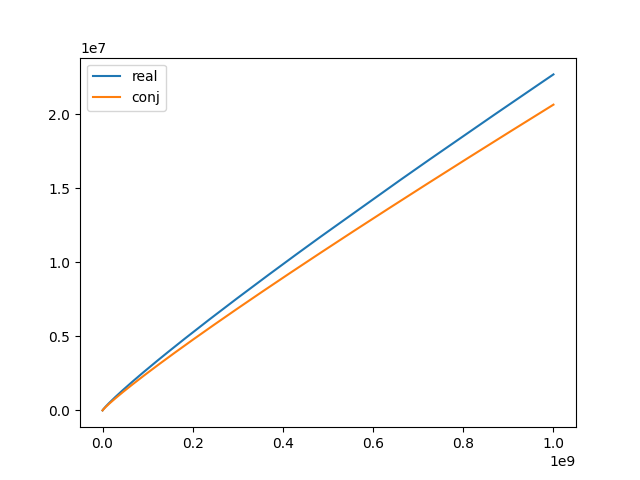}
  \caption{$c_1=1.03, c_2 = 1.01$}
  \end{minipage}
  \caption{Comparison of the conjecture and actual values}
\end{figure}

We also provide a three dimensional figure to show the error terms. The $x$-axis is the value of $c_1$ from $1$ to $1.2$ while the $y$-axis is the value of $c_2$ form $1$ to $1.2$. The $z$-axis stands for the error when $x=10^9$. 

\begin{figure}[htb]
 \centering
  \includegraphics[width=\linewidth]{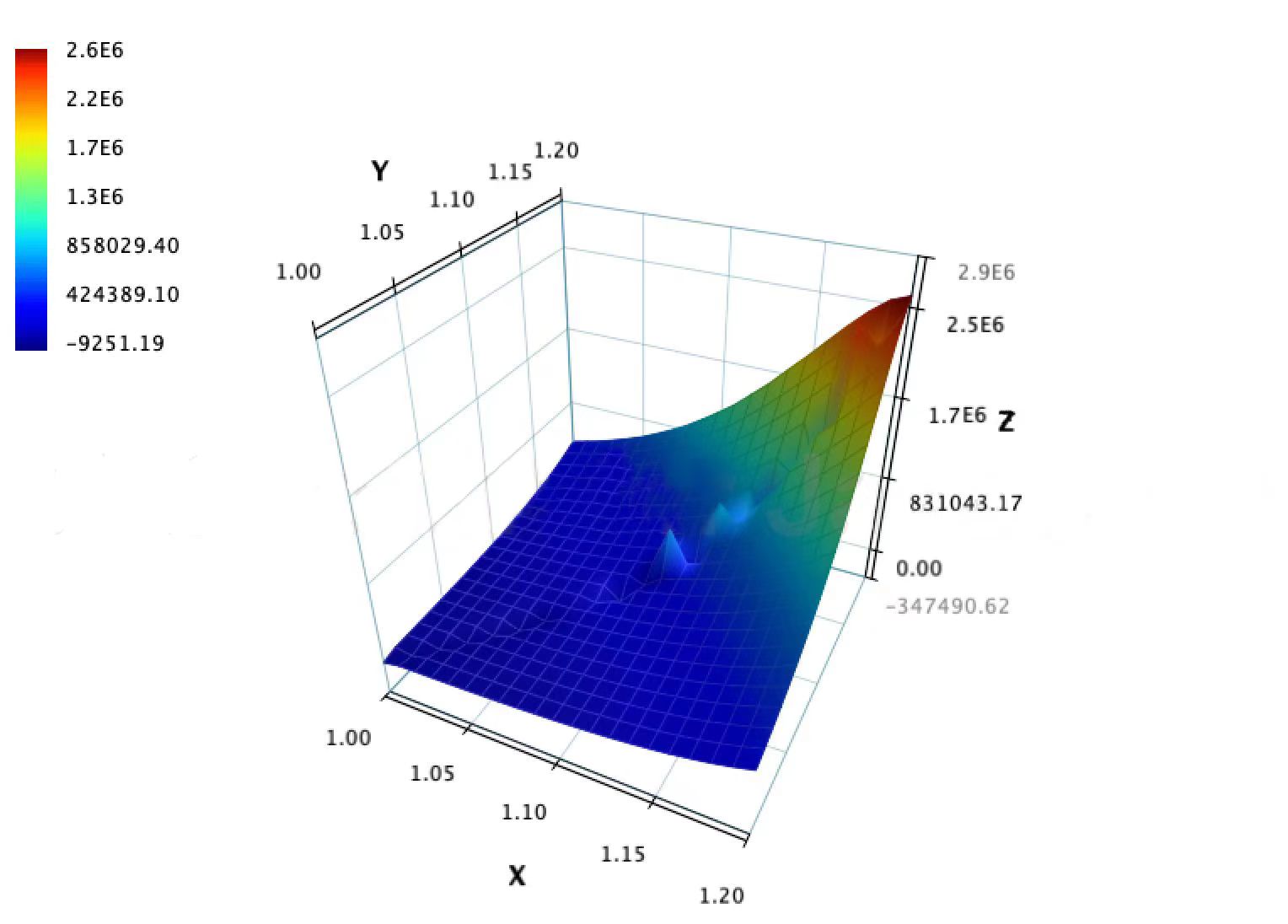}
  \caption{The change of the error term for different $c_1$ and $c_2$}
\end{figure}
 
From the figures, it is not hard to see that Conjecture \ref{conj:main} works fine when $c_1$ and $c_2$ are small. However, the error is not as good for larger $c_1$ and $c_2$. This is because our method does not give a secondary term of the conjecture. We give the detailed reason why our method is not able to provide a secondary term in Section \ref{sec:1.5}. This is also an open problem that can be considered independently. 

\section*{Acknowledgement}
We thank the referee for many valuable suggestions and comments. We also thank an anonymous referee for pointing out that case $c_1=c_2$ is not applicable to our conjecture. We appreciate William Banks, Yuchen Ding, Lingyu Guo, Jinjiang Li and Wenguang Zhai contributed useful conversations and comments to us.  

This work was supported by the National Natural Science Foundation of China (No. 11901447, 12271422), the Natural Science Foundation of Shaanxi Province (No. 2024JC-YBMS-029). and the Shaanxi Fundamental Science Research Project for Mathematics and Physics (No. 22JSY006).

\end{document}